\documentclass{article}
\usepackage[left=3cm,right=3cm,top=2.5cm,bottom=2.5cm]{geometry} 

\usepackage[utf8]{inputenc}

\usepackage{amsmath}
\usepackage{amssymb}
\usepackage{amsthm}
\usepackage{graphicx}
\usepackage{mathrsfs}
\usepackage{tikz}
\usepackage{url}

\newtheorem{thm}{Theorem}[section]
\newtheorem{prop}[thm]{Proposition}
\newtheorem{rmk}[thm]{Remark}

\newtheorem*{ques}{Question}

\newtheorem{cor}[thm]{Corollary}
\newtheorem{lemma}[thm]{Lemma}
\newtheorem{exam}[thm]{Example}

\newcommand{\ZZ}{\mathbb{Z}}

\newcommand{\RR}{\mathbb{R}}
\newcommand{\CC}{\mathbb{C}}

\newcommand{\PP}{\mathbb{P}}

\begin{document}

\title{Type \(C\) Richardson Boundaries and Newton--Okounkov Degenerations of Odd-Dimensional Projective Space}
\author{Zhaoyang Liu\thanks{Department of Mathematics, King's College London, London WC2R 2LS, United Kingdom.}}
\date{}
\maketitle

\begin{abstract}
We compare two mirror-theoretic degeneration pictures attached to odd-dimensional projective space. If \(\PP^{2n-1}\) is viewed as a toric variety, one obtains the classical toric mirror. If it is viewed as the type \(C_n\) homogeneous space \(Sp_{2n}/P_1\), Rietsch's Lie-theoretical construction gives a superpotential on the dual side. We compute the type \(C_n\) boundary \(D_C\), while on the dual-side Lusztig torus we compute the corresponding Laurent polynomial. We construct Newton--Okounkov bodies, and in each case the degree-augmented value semigroup is identified with the lattice-point semigroup of the cone over the polar dual of the corresponding Newton polytope. This gives two toric degeneration pictures attached to the same projective space. We also exhibit a rank-one weight degeneration connecting \(D_C\) with the standard toric boundary while keeping the ambient projective space fixed.
\end{abstract}


\section{Introduction}

The mirror of a Fano variety is usually written as a Landau--Ginzburg model \((Y,W)\), where \(Y\) is a non-compact algebraic variety and \(W:Y\to\CC\) is the superpotential. In the simplest cases \(Y\) is an algebraic torus and \(W\) is a Laurent polynomial. The exponents of the monomials of \(W\) form a Newton polytope, and this polytope is one way in which the mirror remembers the degeneration behaviour of the original Fano variety.

For toric Fano varieties this picture is especially concrete. The mirror is the torus whose character lattice is the lattice of one-parameter subgroups of the original torus, and the superpotential is a sum of monomials corresponding to the toric boundary divisors. This point of view goes back at least to Givental's work on mirror symmetry and Gromov--Witten theory for toric varieties and complete intersections \cite{Givental,Givental98}; it also appears in the physics construction of Hori and Vafa \cite{HV}. For projective space the resulting Laurent polynomial is the familiar toric superpotential.

There is another source of Landau--Ginzburg mirrors, coming from representation theory rather than toric geometry. For a homogeneous space \(G/P\), Rietsch constructs a Lie-theoretical Landau--Ginzburg model on an open Richardson stratum in the Langlands dual flag variety \cite{Rietsch}. An open Richardson stratum is an explicit open subvariety obtained as an intersection of opposite Bruhat cells; the notation is recalled in Section~2. To compute the superpotential one usually chooses a torus chart on this open stratum. In this paper these charts are supplied by Lusztig's parametrizations from reduced words. They give dense algebraic tori, called Lusztig tori, on which Rietsch's superpotential becomes an ordinary Laurent polynomial. The precise construction is given in Section~3.2.

The starting point of this paper is that the same projective space has two different mirror-theoretic descriptions. Throughout, we assume \(n\ge2\) and set \(N=2n-1\). First, \(\PP^N\) is a toric variety. It carries the standard toric boundary
\[
D_{\mathrm{tor}}=\left\{\prod_{i=0}^Np_i=0\right\},
\]
and the toric superpotential \(W_{\mathrm{tor}}^{(N)}\). Second, since every line in a symplectic vector space is isotropic, the same projective space can be written as the type \(C_n\) homogeneous space \(Sp_{2n}/P_1\). This second presentation leads to a different natural boundary \(D_C\) on the original projective space, namely the complement of the type \(C_n\) open Richardson locus used in Section~3. On the dual side, the Rietsch mirror is compactified by the odd quadric \(SO_{2n+1}/P_1^\vee\simeq Q^{2n-1}\), and the corresponding superpotential becomes explicit after restricting to a Lusztig torus.

The aim is to compare these two descriptions through anticanonical divisors, Laurent polynomials, and Newton--Okounkov degenerations. Here an anticanonical boundary means a divisor \(D\) such that \(K_{\PP^N}+D\sim0\), so in projective space it is cut out by a homogeneous polynomial of degree \(N+1\). The toric degenerations considered here are produced from flag valuations. Such a valuation records successive orders of vanishing along a nested sequence of subvarieties. Anderson's construction turns a finitely generated value semigroup into a flat degeneration to a toric variety \cite{Anderson}. We use this together with the general theory of Newton--Okounkov bodies \cite{KK,LM}.

The first main result gives the two explicit calculations needed later: the type \(C_n\) boundary and the dual-side Lusztig-torus superpotential. The boundary consists of the two coordinate hyperplanes \(p_0=0\) and \(p_N=0\), together with \(n-1\) explicit quadratic factors. The superpotential is written in the coordinates on the Lusztig torus, denoted \(a_1,\ldots,a_{n-1},b,c_{n-1},\ldots,c_1\).

\begin{thm}\label{thm:intro-boundary-potential}
Let \(N=2n-1\) and write homogeneous coordinates on \(\PP^N\) as \([p_0:\cdots:p_N]\).
For \(1\le r\le n-1\), define
\[
F_r(p_0,\dots,p_N):=\sum_{j=0}^r(-1)^j p_{r-j}\,p_{N-r+j}.
\]
The type \(C_n\) boundary is the reduced anticanonical divisor
\[
D_C
:=\operatorname{div}\bigl(p_0\,p_N\,F_1\,F_2\,\cdots\,F_{n-1}\bigr)_{\mathrm{red}} .
\]
On the dual-side Lusztig torus associated with the reduced word \((1,2,\dots,n-1,n,n-1,\dots,1)\), the Lie-theoretical superpotential is
\[
W_{\mathrm{Lus}}^{(2n-1)}(a_1,\dots,a_{n-1},b,c_{n-1},\dots,c_1)
:=
a_1+\cdots+a_{n-1}+b+c_{n-1}+\cdots+c_1
+
\frac{q\,(a_1+c_1)}{a_1\cdots a_{n-1}\,b^2\,c_{n-1}\cdots c_1}.
\]
\end{thm}

When \(n=2\), this Lie-theoretical mirror calculation for \(\PP^3\) follows the unpublished manuscript of Tillmann-Morris \cite{TillmannMorris}. Section~6 of that manuscript also states the odd-dimensional formula for the Lusztig-torus Laurent polynomial that appears in Theorem~\ref{thm:intro-boundary-potential}. Section~4.2 uses the \(\PP^3\) case as a toy model before passing to general odd-dimensional projective space.

The second main result constructs a Newton--Okounkov flag on \(\PP^N\) adapted to \(D_C\), with the first steps of the flag following the quadratic factors appearing in the boundary. A section of \(\mathcal O((N+1)k)\) is normalized by the \(k\)-th power of the anticanonical section \(H_N\); on the valuation side this normalization appears as subtracting \(k\mathbf 1\). The resulting lattice points are then compared with the exponent vectors of the Lusztig-torus Laurent polynomial. The raw flag order is kept fixed; a lattice permutation is used only for this comparison.

\begin{thm}\label{thm:intro-no-body}
Let \(N=2n-1\). There exists an admissible flag \(Y_\bullet\) on \(\PP^N\), with associated flag valuation \(\nu_{\mathrm{flag}}\), such that for every \(k\ge0\),
\[
\sigma\left(
\left\{
\nu_{\mathrm{flag}}(s)-k\mathbf 1:
0\ne s\in H^0\bigl(\PP^N,\mathcal O((N+1)k)\bigr)
\right\}
\right)
=
kP_n^\vee\cap\ZZ^{2n-1}.
\]
Here \(H_N=p_0p_N\prod_{r=1}^{n-1}F_r\), \(\nu_{\mathrm{flag}}(H_N)=\mathbf 1=(1,\ldots,1)\), and \(\sigma\) and \(P_n^\vee\) are defined in Section~4.
\end{thm}

\begin{cor}\label{cor:intro-toric-degen}
Let \(S_{\mathrm{flag}}\) be the degree-augmented value semigroup of the flag valuation in Theorem~\ref{thm:intro-no-body}. Under the unimodular lattice automorphism \(\Sigma\), it is identified with
\[
S_n=\{(k,m):k\ge0,\ m\in kP_n^\vee\cap\ZZ^{2n-1}\}.
\]
The Rees algebra of the flag-valuation filtration gives a flat degeneration
\[
\PP^{2n-1}\;\rightsquigarrow\; \operatorname{Proj}\CC[S_n].
\]
The fan of this toric variety is the normal fan of \(P_n^\vee\), equivalently the face fan of \(\operatorname{Newt}(W_{\mathrm{Lus}}^{(2n-1)})\).
\end{cor}

These data lead to three degeneration pictures. The standard toric coordinate valuation gives the toric degeneration \(\PP^N\rightsquigarrow \PP^N\), which is trivial as a degeneration of toric varieties but useful as a comparison. The type \(C_n\) flag valuation gives the toric degeneration controlled by \(P_n^\vee\). Finally, a rank-one weight degeneration sends the pair \((\PP^N,D_C)\) to \((\PP^N,D_{\mathrm{tor}})\) while keeping the ambient projective space fixed.

In type \(A\), Rietsch and Williams related cluster valuations, Newton--Okounkov bodies, superpotential polytopes, and toric degenerations for Grassmannians \cite{RW}. Related cluster-duality phenomena for Lagrangian and orthogonal Grassmannians were studied by Wang \cite{clusterC}, who compared a cluster Newton--Okounkov body with a superpotential polytope up to unimodular equivalence. Although this does not directly identify the flag valuation used here, it motivates the following question.

\begin{ques}
Does the flag valuation of Theorem~\ref{thm:intro-no-body}, or its value semigroup after the lattice identification \(\Sigma\), admit a cluster-theoretic interpretation? Can the corresponding toric degeneration be described in cluster-theoretic terms?
\end{ques}

The paper is organized as follows. Section~2 recalls the standard toric mirror of projective space and Rietsch's mirror construction. Section~3 computes \(D_C\) and the Lusztig-torus superpotential. Section~4 compares the standard toric valuation with the type \(C_n\) flag valuation and proves the toric and rank-one degenerations.

{\small\noindent\textbf{Acknowledgements.} The author thanks Konstanze Rietsch for her guidance and encouragement, Hannah Tillmann-Morris for her unpublished manuscript, and Changzheng Li for helpful discussions. AI language tools were used only for writing assistance and editorial polishing. The author is responsible for the mathematical content, accuracy, integrity, originality, and final form of the paper. This work was supported by the London School of Geometry and Number Theory Centre for Doctoral Training (CDT), a joint venture between University College London, Imperial College London and King's College London.}

\section{Mirror Constructions}

Fix \(n\ge2\) and put \(N=2n-1\).

\subsection{The standard toric mirror}

Let \(\PP^N\) have homogeneous coordinates \([p_0:\cdots:p_N]\). The standard
toric boundary is
\[
D_{\mathrm{tor}}
=
\sum_{i=0}^N\{p_i=0\}
=
\left\{\prod_{i=0}^N p_i=0\right\},
\]
an anticanonical divisor because it has degree \(N+1\). On the affine chart
\(p_0\ne 0\), write \(x_i=p_i/p_0\) for \(1\leq i\leq N\).

\begin{exam}\label{ex:tor-PN}
For the standard toric structure on \(\PP^N\) \cite{Bat93,Givental,Givental98,HV}, the fan rays are
\(e_1,\ldots,e_N\) and \(-e_1-\cdots-e_N\). After the one-parameter
specialization, the toric superpotential is
\[
W_{\mathrm{tor}}^{(N)}(x_1,\ldots,x_N)
=
x_1+\cdots+x_N+\frac{q}{x_1\cdots x_N}.
\]
The compactifying boundary of this mirror is the toric boundary
\(D_{\mathrm{tor}}\).
\end{exam}

Let \(P_{\mathrm{tor}}^{(N)}=\operatorname{Newt}(W_{\mathrm{tor}}^{(N)})\). Then
\[
P_{\mathrm{tor}}^{(N)}
=
\operatorname{Conv}\{e_1,\ldots,e_N,\,-e_1-\cdots-e_N\}
\subset \RR^N.
\]
With the convention
\[
P^\vee=\{u:\langle u,v\rangle\ge -1\ \text{for all }v\in P\},
\]
its polar dual is
\[
\bigl(P_{\mathrm{tor}}^{(N)}\bigr)^\vee
=
\{u=(u_1,\ldots,u_N)\in\RR^N:
u_i\ge -1\ \text{for all }i,\ 
u_1+\cdots+u_N\le1\}.
\]

\subsection{Rietsch's Lie-theoretical mirror}

All varieties are over \(\CC\). We fix the notation from Rietsch's Lie-theoretic mirror construction. Let \(G\) be a connected simply connected semisimple algebraic group. Fix opposite Borel subgroups \(B_+\) and \(B_-\), with unipotent radicals \(U_+\) and \(U_-\), and put
\[
T=B_+\cap B_-.
\]
The Weyl group is
\[
W=N_G(T)/T,
\]
with longest element \(w_0\). For \(w\in W\), we choose a representative in \(G\), denoted by \(\dot w\).

Let \(\mathfrak g,\mathfrak b_\pm,\mathfrak u_\pm,\mathfrak h\) be the Lie algebras of \(G,B_\pm,U_\pm,T\). Let \(\Delta_+\) be the set of positive roots determined by \(B_+\). We write \(I\) for the index set of simple roots \(\{\alpha_i:i\in I\}\). For \(i\in I\), choose Chevalley generators
\[
e_i\in\mathfrak g_{\alpha_i},
\qquad
f_i\in\mathfrak g_{-\alpha_i}.
\]
They define one parameter subgroups
\[
x_i(a)=\exp(ae_i),
\qquad
y_i(a)=\exp(af_i).
\]
We take
\[
\dot s_i=x_i(1)y_i(-1)x_i(1)
\]
as the representative of the simple reflection \(s_i\).

Let \(P\supset B_-\) be a parabolic subgroup. Define
\[
I_P=\{i\in I:\dot s_i\in P\},
\qquad
I^P=I\setminus I_P.
\]
The parabolic Weyl subgroup is
\[
W_P=\langle s_i:i\in I_P\rangle,
\]
with longest element \(w_P\). We also use
\[
W^P=\{w\in W:\ell(ws_i)>\ell(w)\text{ for all }i\in I_P\}
\]
for the minimal length representatives of \(W/W_P\). If
\[
I^P=\{n_1,\ldots,n_k\},
\qquad
n_1<\cdots<n_k,
\]
then \(k=\dim H^2(G/P,\CC)\).

Let \(G^\vee\) be the Langlands dual group. Since \(G\) is simply connected, \(G^\vee\) is adjoint. We use the notation \(B_\pm^\vee,U_\pm^\vee,T^\vee,\mathfrak g^\vee\), and Chevalley generators \(e_i^\vee,f_i^\vee\). The Weyl group is again \(W\), and we keep the dot notation \(\dot w\) for representatives in \(G^\vee\) when no confusion can occur.

For \(v\leq w\) in the Bruhat order, the open Richardson variety is the intersection of a Schubert variety and an opposite Schubert variety in \(G^\vee/B_-^\vee\):
\[
\mathcal R_{v,w}^\vee
=
\left(B_+^\vee\dot vB_-^\vee\cap B_-^\vee\dot wB_-^\vee\right)/B_-^\vee
\subset G^\vee/B_-^\vee .
\]
It is smooth and irreducible of dimension \(\ell(w)-\ell(v)\). The open Richardson variety relevant to \(G/P\) is \(\mathcal R_{w_P,w_0}^\vee\). Let \(\pi:G^\vee/B_-^\vee\to G^\vee/P^\vee\) be the natural projection. Its image is a projected Richardson stratum in the sense of Knutson--Lam--Speyer \cite{KLS}. We denote the reduced complement of this image by
\[
D_{\mathrm{Lie}}^\vee=G^\vee/P^\vee\setminus \pi(\mathcal R_{w_P,w_0}^\vee).
\]
The divisor \(D_{\mathrm{Lie}}^\vee\) is the Rietsch boundary, namely the boundary selected by the Lie-theoretical mirror.

The parameters come from the small quantum cohomology ring of \(G/P\), which is a deformation of \(H^*(G/P)\) with
\[
k=\dim H^2(G/P,\CC)
\]
parameters \(q_1,\ldots,q_k\). These parameters correspond to a basis of curve classes dual to the Schubert divisor classes. In the above Lie-theoretic notation, they are indexed by
\[
I^P=\{n_1,\ldots,n_k\}.
\]

Rietsch identifies the nonzero quantum parameter space with the torus \((T^\vee)^{W_P}\). More precisely,
\[
(T^\vee)^{W_P}=\{t\in T^\vee:\alpha_i^\vee(t)=1\ \text{for all }i\in I_P\},
\]
and the coordinates \(q_j\) are identified with the characters \(\alpha_{n_j}^\vee\) on this torus. Thus working over \((T^\vee)^{W_P}\) is the same as inverting the quantum parameters.


For fixed \(t\in (T^\vee)^{W_P}\), we have 
\[
Z_P^t=B_-^\vee\cap U_+^\vee\,t\dot w_P\dot w_0^{-1}U_+^\vee .
\]
These fibres form a family
\[
Z_P=
\{(t,b)\in (T^\vee)^{W_P}\times B_-^\vee:
b\in U_+^\vee\,t\dot w_P\dot w_0^{-1}U_+^\vee\}
\]
over \((T^\vee)^{W_P}\) by projection to the first factor. 

Rietsch proves that \(Z_P^t\) is smooth of dimension \(\dim G/P\), and the map
\[
Z_P^t\longrightarrow \mathcal R_{w_P,w_0}^\vee,\qquad
b\longmapsto b\dot w_0B_-^\vee/B_-^\vee
\]
is an isomorphism. Hence each fibre has a natural compactification by \(G^\vee/P^\vee\), with boundary \(D_{\mathrm{Lie}}^\vee\).

Define \(F=\sum_{i\in I}(e_i^\vee)^*\in(\mathfrak g^\vee)^*\), where \(e_i^\vee\) are the Chevalley generators fixed above. Let \(\rho\) denote the sum of the fundamental coweights. By definition, each \(b\in Z_P^t\) has a factorization
\[
b=u_1t\dot w_P\dot w_0^{-1}u_2^{-1},
\qquad u_1,u_2\in U_+^\vee,
\]
and Rietsch's superpotential \cite{Rietsch} is
\[
\mathcal F_P(t,b)=F(u_2\cdot\rho)-F(u_1\cdot\rho).
\]
This expression is equivalent to the representation-theoretic definition and gives a regular function on \(Z_P\).

The Rietsch mirror of \(G/P\) is the Landau--Ginzburg model \((Z_P,\mathcal F_P)\to (T^\vee)^{W_P}\). After fixing the quantum parameters \(t\), we can also write it as
\[
\left(X_P^\vee,\mathcal W_t\right)
=
\left(G^\vee/P^\vee\setminus D_{\mathrm{Lie}}^\vee,\ \mathcal F_P|_{Z_P^t}\right),
\]
where \(Z_P^t\simeq X_P^\vee\) through \(b\mapsto b\dot w_0B_-^\vee\). The critical locus of \(\mathcal F_P\) along the fibres recovers Peterson's presentation of the quantum cohomology of \(G/P\) with the quantum parameters inverted \cite{Peterson,LS}.

For the type \(C_n\) model considered below, the original homogeneous space is
\[
Sp_{2n}/P_1\simeq \PP^{2n-1}.
\]
The Rietsch mirror for this homogeneous space is defined on the dual odd quadric
\[
SO_{2n+1}/P_1^\vee\simeq Q^{2n-1}.
\]
The fibre \(Z_P^t\) is identified with the open Richardson variety \(\mathcal R_{w_P,w_0}^\vee\), and compactifying this open variety inside \(G^\vee/P_1^\vee\) produces the dual-side boundary
\[
D_{\mathrm{Lie}}^\vee
=
G^\vee/P_1^\vee\setminus \pi(\mathcal R_{w_P,w_0}^\vee).
\]
Following the philosophy of the Gross--Siebert program, the mirror map should be viewed as a relation between log Calabi--Yau pairs, not only between the ambient spaces. Thus, on the original projective space \(\PP^N\), we also use the open Richardson variety attached to the type \(C_n\) homogeneous model \(Sp_{2n}/P_1\). Its complement is the type \(C_n\) boundary \(D_C\subset\PP^N\) computed below. This construction follows Rietsch's Richardson-stratum viewpoint and the notation used by Pech, Rietsch, and Williams~\cite{Rietsch,PRW}.

\section{Anticanonical Divisors}
Put \(N=2n-1\). Recall from Section~2.1 that the standard toric anticanonical boundary is
\[
D_{\mathrm{tor}}=\sum_{i=0}^{N}\{p_i=0\},
\]
and that its mirror Laurent polynomial is \(W_{\mathrm{tor}}^{(N)}\). We now construct the type \(C_n\) boundary \(D_C\) on \(Sp_{2n}/P_1\simeq\PP^{2n-1}\).

\subsection{The type \(C\) boundary}
Let \(G=Sp_{2n}\), and let \(P=P_1\) be the maximal parabolic corresponding to the first node. Since every line in a symplectic vector space is isotropic,
\[
G/P_1\simeq\PP^{2n-1}.
\]

Precisely, for \(V=\CC^{N+1}\) with ordered basis \(e_0,\ldots,e_N\), where \(N=2n-1\), let
\[
J_{i,N-i}=(-1)^i,
\qquad
J_{ij}=0\quad\text{if }i+j\ne N
\]
be the skew matrix.
We have \(G=Sp(V,J)=\{g\in GL(V):g^tJg=J\}\). \footnote{Here we use row-vector conventions.} Let \(P=P_1\) be the stabilizer of the row line \(\CC e_N^t\). Then
\[
P\backslash G\simeq \PP^N,
\qquad
Pg\longmapsto [e_N^tg]=[p_0:\cdots:p_N].
\]
Let \(B_+\), \(B_-\), \(U_+\), \(U_-\) be the upper/lower triangular Borel and unipotent subgroups in this basis. For \(P=P_1\), the parabolic Weyl subgroup is generated by \(s_2,\ldots,s_n\). Choose representatives with \(\dot w_0=J\), and set \(M=\dot w_P\dot w_0^{-1}\). With these conventions, \(M\) is the signed permutation matrix
\[
M_{0N}=-1,
\qquad
M_{N0}=1,
\qquad
M_{ii}=-1\quad(1\le i\le N-1),
\]
all other entries being zero.

Let \(S=B_-\cap U_+MU_+\). The map
\[
S\longrightarrow (B_+\dot w_PB_-\cap B_-\dot w_0B_-)/B_-,
\qquad
b\longmapsto b\dot w_0B_-/B_-,
\]
identifies \(S\) with the open Richardson variety. This Richardson variety belongs to the type \(C_n\) flag variety for the original group and is used only to define the type \(C_n\) boundary \(D_C\); it should not be confused with the dual-side Richardson variety appearing in Rietsch's mirror construction.
Composing with \(P\backslash G\simeq\PP^N\), define
\[
\eta:S\longrightarrow \PP^N,
\qquad
b\longmapsto [e_N^tb\dot w_0].
\]

\begin{prop}\label{prop:M-matrix}
With the conventions above, \(M:=\dot w_P\dot w_0^{-1}\) is the signed permutation matrix with nonzero entries
\[
M_{0N}=-1,\qquad
M_{N0}=1,\qquad
M_{ii}=-1\quad(1\le i\le N-1),
\]
and all other entries zero.
\end{prop}

\begin{proof}
The representative \(\dot w_0 = J\) acts on the ordered basis by
\[
\dot w_0(e_i)=J(e_i)=(-1)^i e_{N-i}\qquad(0\le i\le N),
\]
and its inverse is \(\dot w_0^{-1}=J^{-1}=-J\) because \(J^2=-I\) (note that \(N=2n-1\) is odd).

The parabolic Weyl subgroup of \(P=P_1\) is \(W_P=\langle s_2,\dots,s_n\rangle\); its longest element \(w_P\) fixes the one-dimensional weight spaces corresponding to the extremal vectors \(e_0\) and \(e_N\), and on the middle block \(\operatorname{Span}\{e_1,\dots,e_{N-1}\}\) it acts as the longest element of type \(C_{n-1}\).  In the standard signed-permutation model this action is
\[
\dot w_P(e_0)=e_0,\qquad
\dot w_P(e_N)=e_N,\qquad
\dot w_P(e_i)=(-1)^{\,i+1}e_{N-i}\quad(1\le i\le N-1).
\]

Multiplying \(\dot w_P\) by \(\dot w_0^{-1}\) gives the stated signed permutation matrix \(M=\dot w_P\dot w_0^{-1}\).
\end{proof}

\begin{exam}\label{ex:n3-matrices}
For \(n=3\) we have \(N=2n-1=5\) and the matrices act on \(\CC^6\) with basis \(e_0,\dots,e_5\).
The symplectic form is \(J\) with \(J_{i,5-i}=(-1)^i\); its matrix is
\[
\dot w_0 = J =
\begin{pmatrix}
 0& 0& 0& 0& 0& 1\\
 0& 0& 0& 0&-1& 0\\
 0& 0& 0& 1& 0& 0\\
 0& 0&-1& 0& 0& 0\\
 0& 1& 0& 0& 0& 0\\
-1& 0& 0& 0& 0& 0
\end{pmatrix}.
\]
The parabolic Weyl subgroup is \(W_P=\langle s_2,s_3\rangle\) of type \(C_2\); its longest element acts by \(\dot w_P(e_0)=e_0\), \(\dot w_P(e_5)=e_5\), and \(\dot w_P(e_i)=(-1)^{i+1}e_{5-i}\) for \(1\le i\le 4\), giving
\[
\dot w_P =
\begin{pmatrix}
1& 0& 0& 0& 0& 0\\
0& 0& 0& 0&-1& 0\\
0& 0& 0& 1& 0& 0\\
0& 0&-1& 0& 0& 0\\
0& 1& 0& 0& 0& 0\\
0& 0& 0& 0& 0& 1
\end{pmatrix}.
\]
Multiplying \(\dot w_P\) with \(\dot w_0^{-1}=-J\) yields
\[
M = \dot w_P\dot w_0^{-1} =
\begin{pmatrix}
 0& 0& 0& 0& 0&-1\\
 0&-1& 0& 0& 0& 0\\
 0& 0&-1& 0& 0& 0\\
 0& 0& 0&-1& 0& 0\\
 0& 0& 0& 0&-1& 0\\
 1& 0& 0& 0& 0& 0
\end{pmatrix},
\]
recovering the pattern \(M_{05}=-1\), \(M_{50}=1\) and \(M_{ii}=-1\) for \(i=1,2,3,4\).
\end{exam}

For a row vector \(x=(x_0,\ldots,x_N)\), put
\[
F_r(x)=\sum_{j=0}^r(-1)^jx_{r-j}x_{N-r+j}
\qquad(0\le r\le n-1).
\]
Thus \(F_0(x)=x_N\). In homogeneous coordinates \([p_0:\cdots:p_N]\), we write \(F_r=F_r(p)\) for \(1\le r\le n-1\).

\begin{thm}\label{thm:richardson-boundary}
The type \(C_n\) open locus in \(\PP^N=\PP^{2n-1}\) is
\[
\eta(S)=
\{[p_0:\cdots:p_N]\in \PP^N:
 p_0p_NF_1\cdots F_{n-1}\ne0\}.
\]
Consequently the reduced type \(C_n\) boundary is
\[
D_C
=
\operatorname{div}\bigl(p_0p_NF_1\cdots F_{n-1}\bigr)_{\mathrm{red}}
=
\{p_0=0\}
\cup
\left(\bigcup_{r=1}^{n-1}\{F_r=0\}\right)
\cup
\{p_N=0\}.
\]
\end{thm}

The proof of Theorem~\ref{thm:richardson-boundary} relies on the following criterion for Gauss factorization.
\begin{lemma}\label{lem:southeast-elimination}
Let \(A\in GL(V)\) where \(V\) is an \((N+1)\)-dimensional vector space with ordered basis \(e_0,\dots,e_N\).
Define the southeast principal minors
\[
\Delta_k(A)=
\det A_{\{k,k+1,\dots,N\},\{k,k+1,\dots,N\}}
\qquad(1\le k\le N),
\]
and set \(\Delta_{N+1}(A)=1\).
Then \(A\) admits a unique Gauss factorization \(A=u^{-1}b\), with \(u\in U_+\) and \(b\in B_-\), if and only if \(\Delta_k(A)\neq 0\) for \(1\le k\le N\).
Moreover, if \(A\in Sp(V,J)\), then both factors \(u\) and \(b\) belong to \(Sp(V,J)\).
\end{lemma}

\begin{proof}
This is the usual LU criterion applied from the lower-right corner.
Proceed by descending induction on the row index \(k=N,N-1,\dots,1\).
At step \(k\), after rows \(k+1,\dots,N\) have been reduced, the Schur complement of the southeast block with rows and columns \(\{k,\dots,N\}\) is the scalar
\[
\frac{\Delta_k(A)}{\Delta_{k+1}(A)}.
\]
This quantity is the coefficient needed to eliminate the entries in row \(k\) above the diagonal; it is well defined and nonzero precisely when \(\Delta_k(A)\neq 0\).
The row operation that clears column \(k\) is uniquely determined, and the resulting matrix has the required lower triangular form.
Induction yields the unique factor \(u\in U_+\) with \(uA=b\in B_-\).

For the symplectic part, suppose \(A\in Sp(V,J)\), i.e.\ \(A^tJA=J\).
The involution \(\iota\colon g\mapsto J^{-1}(g^{-1})^tJ\) satisfies \(\iota(U_+)=U_+\) and \(\iota(B_-)=B_-\), and it fixes \(A\) because \(A^tJA=J\) implies \(J^{-1}(A^{-1})^tJ=A\).
Applying \(\iota\) to the factorization \(A=u^{-1}b\) gives \(A=\iota(u)^{-1}\iota(b)\).
Since \(\iota(u)\in U_+\) and \(\iota(b)\in B_-\), uniqueness of the Gauss factorization forces \(\iota(u)=u\) and \(\iota(b)=b\), i.e.\ \(u,b\in Sp(V,J)\).
\end{proof}

We use Jacobi's complementary minor identity; see \cite[Chapter I, \S4]{Gantmacher}.
\begin{lemma}\label{lem:jacobi-complementary-minor}
Let \(A\in GL_{N+1}(\CC)\), with rows and columns indexed by \(0,\ldots,N\). Let \(I,J\subset\{0,\ldots,N\}\) have the same cardinality, and write \(I^c,J^c\) for their complements, each in increasing order. Then
\[
\det A_{I,J}
=
(-1)^{\sum_{i\in I}i+\sum_{j\in J}j}
\det(A)\,
\det (A^{-1})_{J^c,I^c}.
\]
\end{lemma}

\begin{lemma}\label{lem:symplectic-inverse-last-column}
Let \(v\in U_+\cap Sp(V,J)\), and write its first row as
\[
q=e_0^tv=(q_0,q_1,\ldots,q_N).
\]
Then the last column of \(v^{-1}\) is obtained from the first row of \(v\) by reversing the order and alternating signs:
\[
(v^{-1})_{iN}=(-1)^{i+1}q_{N-i}=(-1)^{i-N}q_{N-i}
\qquad(0\le i\le N).
\]
\end{lemma}

\begin{proof}
Since \(v\in Sp(V,J)\), we have \(v^tJv=J\). Hence
\[
v^{-1}=J^{-1}v^tJ=-Jv^tJ,
\]
because \(J^2=-I\) for \(N=2n-1\). Now \(Je_N=e_0\), so the last column of \(v^{-1}\) is
\[
v^{-1}e_N=-Jv^tJe_N=-Jv^te_0.
\]
But \(v^te_0\) is the transpose of the first row of \(v\), namely \(q^t\). Therefore
\[
v^{-1}e_N=-Jq^t.
\]
Taking the \(i\)-th entry and using \(J_{i,N-i}=(-1)^i\), we get
\[
(v^{-1})_{iN}
=
-J_{i,N-i}q_{N-i}
=
-(-1)^iq_{N-i}
=
(-1)^{i+1}q_{N-i}.
\]
Since \(N\) is odd, \((-1)^{i+1}=(-1)^{i-N}\), giving the displayed formula. The unipotent condition gives in particular \(q_0=1\) and the upper-triangular shape used below.
\end{proof}

\begin{proof}[Proof of Theorem~\ref{thm:richardson-boundary}]
By definition, \(\eta:S\longrightarrow \PP^N,b\longmapsto [e_N^tb\dot w_0]\). Since \(S=B_-\cap U_+MU_+\), every \(b\in S\) can be written as \(b=uMv\) for some \(u,v\in U_+\cap Sp(V,J)\). Let \(q=(q_0,\ldots,q_N):=e_0^tv\) be the first row of \(v\). Since \(v\in U_+\), one has \(q_0=1\). The last row of \(u\) is \(e_N^t\), and the last row of \(M\) is \(e_0^t\). Therefore
\[
e_N^tb=e_N^tuMv=e_N^tMv=e_0^tv=q.
\]
The homogeneous coordinates of \(\eta(b)\) are \(p=qJ\), namely
\[
p_i=(qJ)_i=(-1)^{N-i}q_{N-i}.
\]
Thus \(p_i=(-1)^{N-i}q_{N-i}\). Since \(N=2n-1\) is odd,
\[
p_N=q_0=1,
\qquad
p_0=-q_N.
\]
Moreover, substituting \(p_i=(-1)^{N-i}q_{N-i}\) gives
\[
F_r(p)=-F_r(q)
\qquad
(1\le r\le n-1).
\]
Thus the stated open set in \(p\)-coordinates is the image, under \(q\mapsto [qJ]\), of the \(q\)-coordinate condition
\[
q_0=1,\qquad F_0(q)F_1(q)\cdots F_{n-1}(q)\ne0.
\]
It remains to compare this condition with the Gauss factorization condition for \(Mv\).

For an invertible matrix \(A\), write
\[
\Delta_k(A)=
\det A_{\{k,k+1,\ldots,N\},\{k,k+1,\ldots,N\}},
\qquad
1\le k\le N.
\]
By Lemma~\ref{lem:southeast-elimination}, \(A\) admits a unique Gauss factorization \(A=u^{-1}b\) with \(u\in U_+\) and \(b\in B_-\) if and only if
\[
\Delta_k(A)\ne0\qquad(1\le k\le N).
\]
Now take \(A=Mv\). The rows of \(A\) are
\[ A_0=-v_N,
\quad
A_i=-v_i\quad(1\le i\le N-1),
\quad
A_N=v_0,
\]
where \(v_i\) is the \(i\)-th row of \(v\). Since \(v\in U_+\), all diagonal entries of \(v\) are \(1\) and all entries below the diagonal are \(0\). Thus \(A=Mv\) has the corresponding form. For example, when \(N=5\),
\[
A=
\begin{pmatrix}
0&0&0&0&0&-1\\
0&-1&*&*&*&*\\
0&0&-1&*&*&*\\
0&0&0&-1&*&*\\
0&0&0&0&-1&*\\
1&q_1&q_2&q_3&q_4&q_5
\end{pmatrix}.
\]
For \(1\le k\le N-1\), expand \(\Delta_k(A)\) along its first column, i.e. the global column \(k\). In that column the only possible nonzero entries are \(-1\) in row \(k\) and \(q_k\) in row \(N\). Hence
\[
\Delta_k(A)
=
-\Delta_{k+1}(A)
+(-1)^{N-k}q_kC_k,
\]
where
\[
C_k=
\det A_{\{k,\ldots,N-1\},\{k+1,\ldots,N\}}.
\]
To compute \(C_k\), apply Lemma~\ref{lem:jacobi-complementary-minor} to the matrix \(A\). Put
\[
I=\{k,\ldots,N-1\},
\qquad
J=\{k+1,\ldots,N\}.
\]
Then
\[
I^c=\{0,\ldots,k-1,N\},
\qquad
J^c=\{0,\ldots,k\}.
\]
Since \(A\in Sp(V,J)\), \(\det(A)=1\), and the sign in Jacobi's identity is \((-1)^{N-k}\). Thus
\[
C_k
=
(-1)^{N-k}
\det(A^{-1})_{J^c,I^c}.
\]
Here the complementary row set in \(A\) is \(I^c=\{0,\ldots,k-1,N\}\), and the complementary column set is \(J^c=\{0,\ldots,k\}\); in Jacobi's formula they occur as rows \(J^c\) and columns \(I^c\) of \(A^{-1}\). Since \(A=Mv\), we have \(A^{-1}=v^{-1}M^{-1}\). The inverse of \(M\) is again a signed permutation matrix:
\[
(M^{-1})_{0N}=1,
\qquad
(M^{-1})_{N0}=-1,
\qquad
(M^{-1})_{ii}=-1\quad(1\le i\le N-1),
\]
with all other entries zero. Thus, when \(N=5\), the shapes of \(v^{-1}\) and \(M^{-1}\) are
\[
\begin{array}{c@{\qquad}c}
v^{-1}=
\begin{pmatrix}
1&*&*&*&*&-q_5\\
0&1&*&*&*&q_4\\
0&0&1&*&*&-q_3\\
0&0&0&1&*&q_2\\
0&0&0&0&1&-q_1\\
0&0&0&0&0&1
\end{pmatrix}
&
M^{-1}=
\begin{pmatrix}
0&0&0&0&0&1\\
0&-1&0&0&0&0\\
0&0&-1&0&0&0\\
0&0&0&-1&0&0\\
0&0&0&0&-1&0\\
-1&0&0&0&0&0
\end{pmatrix}.
\end{array}
\]
The displayed last column of \(v^{-1}\) is given by Lemma~\ref{lem:symplectic-inverse-last-column}.
Right multiplication by \(M^{-1}\) sends column \(0\) to minus the \(N\)-th column, sends column \(i\) to minus the \(i\)-th column for \(1\le i\le N-1\), and sends column \(N\) to the \(0\)-th column. Hence \(A^{-1}=v^{-1}M^{-1}\) has the following form when \(N=5\):
\[
A^{-1}=
\begin{pmatrix}
q_5&*&*&*&*&1\\
-q_4&-1&*&*&*&0\\
q_3&0&-1&*&*&0\\
-q_2&0&0&-1&*&0\\
q_1&0&0&0&-1&0\\
-1&0&0&0&0&0
\end{pmatrix}.
\]
For the required minor, this shape shows that the matrix is triangular except for the column coming from the \(N\)-th column of \(v^{-1}\):
\[
\det(A^{-1})_{J^c,I^c}=q_{N-k}.
\]
Indeed, the column set \(I^c\) consists of the columns \(0,\ldots,k-1,N\); after multiplying by \(M^{-1}\), these are \(-\)the \(N,1,\ldots,k-1\) columns of \(v^{-1}\), followed by the \(0\)-th column of \(v^{-1}\). The \(0\)-th column is \(e_0\), the columns \(1,\ldots,k-1\) form a unit triangular block, and the remaining entry is
\[
(v^{-1})_{kN}=(-1)^{k-N}q_{N-k};
\]
the triangular expansion gives
\[
\det(A^{-1})_{J^c,I^c}
=(-1)^{k+1}(v^{-1})_{kN}
=q_{N-k},
\]
because \(N\) is odd. Hence
\[
C_k=(-1)^{N-k}q_{N-k},
\]
and therefore
\[
\Delta_k(A)+\Delta_{k+1}(A)=q_kq_{N-k}
\qquad(1\le k\le N-1).
\]

Now \(\Delta_N(A)=A_{NN}=q_N=F_0(q)\). Starting from this value and applying the recurrence backwards gives
\[
\Delta_{N-s}(A)=H_s(q)
\qquad(0\le s\le N-1),
\]
where
\[
H_s(q)=\sum_{j=0}^s(-1)^jq_{s-j}q_{N-s+j}.
\]
For \(0\le s\le n-1\), this is \(F_s(q)\). For \(s\ge n\), the middle terms in the alternating sum cancel in pairs because \(N=2n-1\) is odd, and \(H_s(q)=F_{N-1-s}(q)\). Thus every southeast principal minor \(\Delta_k(A)=\Delta_k(Mv)\), \(1\le k\le N\), is one of \(F_0(q),\ldots,F_{n-1}(q)\), and each of these \(F_r(q)\)'s occurs. Therefore all southeast principal minors of \(Mv\) are nonzero if and only if \(F_0(q),F_1(q),\ldots,F_{n-1}(q)\) are nonzero.

Conversely, take any row vector \(q=(1,q_1,\ldots,q_N)\) satisfying
\[
F_0(q)F_1(q)\cdots F_{n-1}(q)\ne0.
\]
It remains to realize \(q\) as the first row of an element of \(U_+\cap Sp(V,J)\). Write \(a=(q_1,\ldots,q_{N-1})\) and \(z=q_N\). With respect to the decomposition
\[
V=\CC e_0\oplus \operatorname{Span}(e_1,\ldots,e_{N-1})\oplus \CC e_N,
\]
write
\[
J=\begin{pmatrix}
0&0&1\\
0&J'&0\\
-1&0&0
\end{pmatrix}.
\]
Set
\[
v(q)=
\begin{pmatrix}
1&a&z\\
0&I_{N-1}&J'a^t\\
0&0&1
\end{pmatrix}.
\]
Using \(J'^t=-J'\) and \(J'J'=-I\) for the middle symplectic block, block multiplication gives
\[
v(q)^tJv(q)=
\begin{pmatrix}
0&0&1\\
0&J'&0\\
-1&0&0
\end{pmatrix}
=J.
\]
Thus \(v(q)\in U_+\cap Sp(V,J)\), and its first row is \(q\). Since \(F_0(q),F_1(q),\ldots,F_{n-1}(q)\) are nonzero, all southeast principal minors of \(Mv(q)\) are nonzero. Gaussian elimination gives \(u\in U_+\cap Sp(V,J)\) with
\[
uMv(q)\in B_-.
\]
Then \(b=uMv(q)\in S\), and \(e_N^tb\dot w_0=qJ\). Hence every row vector \(q=(1,q_1,\ldots,q_N)\) satisfying \(F_0(q)F_1(q)\cdots F_{n-1}(q)\ne0\) occurs. By the coordinate identification \(p=qJ\) at the beginning of the proof, this is exactly the reverse inclusion in the stated \(p\)-coordinates.
\end{proof}

\begin{rmk}\label{rmk:square-free-boundary}
The polynomial \(H_N=p_0p_N\prod_{r=1}^{n-1}F_r\) is square-free. Indeed, \(p_0\) and \(p_N\) are distinct linear forms. For \(r\ge1\), the symmetric matrix of the quadratic form \(F_r\) has rank \(2(r+1)\ge4\), so \(F_r\) is irreducible over \(\CC\). Moreover, \(F_r\) is divisible by neither \(p_0\) nor \(p_N\), since it contains the monomial \(p_rp_{N-r}\). Finally, if \(s<r\), the monomial \(p_rp_{N-r}\) occurs in \(F_r\) but not in \(F_s\), so \(F_r\) and \(F_s\) are not proportional. Thus all irreducible factors of \(H_N\) are distinct, and \(D_C\) is reduced.
\end{rmk}

\subsection{The dual odd quadric and the Lusztig torus}
We first fix the notation for the dual odd quadric and for the coordinates used in the Rietsch mirror. Specialize to
\[
G=Sp_{2n}(\CC),
\qquad
P=P_1,
\]
so that \(G/P_1\simeq \PP^{2n-1}\). The Langlands dual group is
\[
G^\vee=SO_{2n+1}(\CC),
\]
Landau--Ginzburg models for odd-dimensional quadrics and their comparison were studied by Pech and Rietsch~\cite{PR}.
The compactification of a Rietsch mirror fibre is
\[
G^\vee/P_1^\vee
=
\operatorname{OG}(1,2n+1)
\cong Q^{2n-1}\subset \PP^{2n},
\]
the smooth odd-dimensional quadric parametrizing isotropic lines in the standard representation of \(SO_{2n+1}\).
The Lie-theoretical Landau--Ginzburg fibre is identified with the open Richardson variety
\[
\mathcal R_{w_P,w_0}^\vee
=
(B_+^\vee\dot w_PB_-^\vee \cap B_-^\vee\dot w_0B_-^\vee)/B_-^\vee
\subset G^\vee/B_-^\vee .
\]
Under the natural projection \(\pi\colon G^\vee/B_-^\vee\to G^\vee/P_1^\vee\), the dual-side Rietsch boundary is
\[
D_{\mathrm{Lie}}^\vee
:=
G^\vee/P_1^\vee\setminus \pi(\mathcal R_{w_P,w_0}^\vee).
\]
The divisor \(D_C\) constructed above is instead the type \(C_n\) boundary on the original projective space.

Let \(N'=2n+1\). We write the standard basis of \(\CC^{N'}\) as \(e_1,\ldots,e_{N'}\), and take the symmetric form with matrix
\[
J^\vee_{i,N'+1-i}=(-1)^{i-1},
\qquad
J^\vee_{ij}=0\quad\text{if }i+j\ne N'+1.
\]
Thus
\[
SO_{2n+1}(\CC)=\{g\in SL_{N'}(\CC):g^tJ^\vee g=J^\vee\}.
\]
Choose Chevalley generators in \(\mathfrak{so}_{2n+1}\) by
\[
e_i^\vee=E_{i,i+1}+E_{N'-i,N'+1-i}
\quad(1\le i\le n-1),
\]
and
\[
e_n^\vee=\sqrt2(E_{n,n+1}+E_{n+1,n+2}).
\]
Set \(x_i(s)=\exp(se_i^\vee)\). For \(i<n\), this is simply \(I+se_i^\vee\), while
\[
x_n(s)=I+\sqrt2s(E_{n,n+1}+E_{n+1,n+2})+s^2E_{n,n+2}.
\]
Let \(\alpha_1^\vee,\ldots,\alpha_n^\vee\) be the simple roots of \((G^\vee,T^\vee,B_+^\vee)\), so that
\[
\mathfrak g^\vee_{\alpha_i^\vee}=\CC e_i^\vee .
\]
With our choice of \(J^\vee\), the diagonal Cartan subalgebra may be written as
\[
\mathfrak h^\vee
=
\left\{
\operatorname{diag}(a_1,\ldots,a_n,0,-a_n,\ldots,-a_1)
\right\}.
\]
In particular,
\[
\rho^\vee
=
\operatorname{diag}(n,n-1,\ldots,1,0,-1,\ldots,-n),
\]
so that \(\alpha_i^\vee(\rho^\vee)=1\) for every simple root and hence
\[
[\rho^\vee,e_i^\vee]=e_i^\vee .
\]
For each \(i\), let \((e_i^\vee)^*\in(\mathfrak g^\vee)^*\) be the linear functional which is \(1\) on \(e_i^\vee\) and vanishes on \(\mathfrak h^\vee\) and on all root spaces other than \(\mathfrak g^\vee_{\alpha_i^\vee}\). For \(u\in U_+^\vee\), define the simple-root coordinate
\[
\varepsilon_i(u)=(e_i^\vee)^*(\operatorname{Ad}_{u^{-1}}\rho^\vee).
\]
Let \(u=I+N\in U_+^\vee\). Since \(U_+^\vee\) is filtered by root height, we have
\[
\operatorname{Ad}_{u^{-1}}\rho^\vee
=
u^{-1}\rho^\vee u
=
\rho^\vee+[\rho^\vee,N]+\text{terms of root height at least }2.
\]
Therefore the coefficient of \(e_i^\vee\) in \(\operatorname{Ad}_{u^{-1}}\rho^\vee\) is exactly the simple-root coefficient of \(u\). In matrix entries this gives
\[
\varepsilon_i(u)=u_{i,i+1}=u_{N'-i,N'+1-i}\quad(1\le i<n),
\]
and, for the short simple root,
\[
\varepsilon_n(u)=\frac{1}{\sqrt2}u_{n,n+1}
=
\frac{1}{\sqrt2}u_{n+1,n+2}.
\]
The parabolic Weyl subgroup is generated by \(s_2,\ldots,s_n\). With the standard representatives one may take
\[
t(q)=\operatorname{diag}(q,1,\ldots,1,q^{-1})\in (T^\vee)^{W_P}
\]
and \(M^\vee=\dot w_P\dot w_0^{-1}\) to be the signed permutation matrix whose nonzero entries are
\[
M^\vee_{1,N'}=1,
\qquad
M^\vee_{i,i}=-1\quad(2\le i\le N'-1),
\qquad
M^\vee_{N',1}=1.
\]

\begin{exam}\label{ex:dual-n2-matrices}
For \(n=2\) we have \(N'=5\), and \(G^\vee/P_1^\vee=SO_5/P_1^\vee\cong Q^3\subset\PP^4\). The symmetric form is
\[
J^\vee=
\begin{pmatrix}
0&0&0&0&1\\
0&0&0&-1&0\\
0&0&1&0&0\\
0&-1&0&0&0\\
1&0&0&0&0
\end{pmatrix}.
\]
For the parabolic subgroup \(W_P=\langle s_2\rangle\), the standard representative of \(w_P=s_2\) may be taken as
\[
\dot w_P=
\begin{pmatrix}
1&0&0&0&0\\
0&0&0&1&0\\
0&0&-1&0&0\\
0&1&0&0&0\\
0&0&0&0&1
\end{pmatrix}.
\]
Since \((J^\vee)^{-1}=J^\vee\), we obtain
\[
M^\vee=\dot w_P(J^\vee)^{-1}
=
\begin{pmatrix}
0&0&0&0&1\\
0&-1&0&0&0\\
0&0&-1&0&0\\
0&0&0&-1&0\\
1&0&0&0&0
\end{pmatrix},
\]
which is the specialization of the displayed formula for \(M^\vee\).
\end{exam}

With these definitions, the Rietsch fibre over the parameter \(q\in\CC^*\) is
\[
Z_{P_1}^{q}
=
B_-^\vee\cap U_+^\vee t(q)M^\vee U_+^\vee .
\]
By the definition of this domain, every point of the fibre can be written in the form
\[
b=u_1t(q)M^\vee u_2,
\qquad
u_1,u_2\in U_+^\vee,
\qquad
b\in B_-^\vee .
\]
This is the polynomial sign convention used below. It is related to Rietsch's original convention as follows: her factorization is
\[
b=u_1^{R}t(q)M^\vee (u_2^{R})^{-1},
\qquad
u_1^{R},u_2^{R}\in U_+^\vee,
\]
and her potential, in terms of the simple-root coordinates defined above, is
\[
\mathcal F_P=
\sum_{i=1}^n\varepsilon_i(u_2^{R})
-
\sum_{i=1}^n\varepsilon_i(u_1^{R}).
\]
In the explicit chart we set \(u_1=u_1^{R}\) and \(u_2=(u_2^{R})^{-1}\). Since inversion changes the simple-root coordinates in \(U_+^\vee\) by a sign, the potential that we compute is
\[
W_{\mathrm{Lie}}=-\mathcal F_P
=
\sum_{i=1}^n\varepsilon_i(u_1)+\sum_{i=1}^n\varepsilon_i(u_2).
\]
It remains to compute the two simple-root-coordinate sums for \(u_1\) and \(u_2\).

Choose \(u_2\) from a Lusztig torus. If \(w=s_{i_1}\cdots s_{i_\ell}\) is a reduced expression in the Weyl group, then the multiplication map
\[
\theta_{\mathbf i}\colon(\CC^*)^\ell\longrightarrow U_+^\vee,
\qquad
(z_1,\ldots,z_\ell)\longmapsto x_{i_1}(z_1)\cdots x_{i_\ell}(z_\ell)
\]
parametrizes a Zariski-open torus in the unipotent Bruhat cell
\[
U_+^\vee(w):=U_+^\vee\cap B_-^\vee\dot wB_-^\vee .
\]
This is the Lusztig torus associated with the reduced word \(\mathbf i\) \cite{Lusztig,MR}. In the present parabolic case, the relevant Weyl group element is \(w_Pw_0=w_Pw_0^{-1}\), and it has the reduced word
\[
\mathbf i=(1,2,\ldots,n-1,n,n-1,\ldots,1)
\]
of length \(2n-1=\dim(G^\vee/P_1^\vee)\). We use the corresponding Lusztig coordinates and set
\[
u_2=
 x_1(a_1)x_2(a_2)\cdots x_{n-1}(a_{n-1})x_n(b)
 x_{n-1}(c_{n-1})\cdots x_1(c_1),
\]
where
\[
(a_1,\ldots,a_{n-1},b,c_{n-1},\ldots,c_1)\in(\CC^*)^{2n-1}.
\]
For this choice of \(u_2\), the factor \(u_1\) is determined by imposing the defining condition \(b\in Z_{P_1}^q\), namely by solving the triangular factorization problem
\[
u_1t(q)M^\vee u_2\in B_-^\vee.
\]
For example, when \(n=2\), so that \(N'=5\), the reduced word is \((1,2,1)\). Writing \(a=a_1\) and \(c=c_1\), the Lusztig product gives
\[
u_2=x_1(a)x_2(b)x_1(c)
=
\begin{pmatrix}
1&a+c&\sqrt2ab&ab^2&ab^2c\\
0&1&\sqrt2b&b^2&b^2c\\
0&0&1&\sqrt2b&\sqrt2bc\\
0&0&0&1&a+c\\
0&0&0&0&1
\end{pmatrix}.
\]

For the computations below, put
\[
A_r=a_1a_2\cdots a_r,
\qquad
A_0=1,
\]
\[
C_r=c_rc_{r+1}\cdots c_{n-1},
\qquad
C_n=1,
\]
\[
D=A_{n-1}b^2C_1
=a_1\cdots a_{n-1}b^2c_1\cdots c_{n-1},
\qquad
s_r=a_r+c_r.
\]

\begin{lemma}\label{lem:u2-contribution}
For the above Lusztig product,
\[
\varepsilon_i(u_2)=a_i+c_i\quad(1\le i\le n-1),
\qquad
\varepsilon_n(u_2)=b.
\]
Moreover, for \(2\le i\le N'-1\) and \(j>i\),
\begin{equation}\label{eq:u2-entry-relation}
(u_2)_{i,j}=\lambda_i (u_2)_{1,j}
\end{equation}
holds, where
\[
\lambda_i=
\begin{cases}
\displaystyle \frac{1}{A_{i-1}}, & 2\le i\le n,\\[1.1em]
\displaystyle \frac{\sqrt2}{A_{n-1}b}, & i=n+1,\\[1.1em]
\displaystyle \frac{s_{N'-i}}{A_{n-1}b^2C_{N'-i}}, & n+2\le i\le 2n.
\end{cases}
\]
\end{lemma}

\begin{proof}
For \(r<n\), we have
\[
x_r(z)=I+ze_r^\vee,
\]
and the only first-superdiagonal entries of \(ze_r^\vee\) are
\[
(ze_r^\vee)_{r,r+1}=z,
\qquad
(ze_r^\vee)_{N'-r,N'+1-r}=z.
\]
For \(r=n\), the formula for \(x_n(b)\) displayed above gives the first-superdiagonal entries
\[
(x_n(b)-I)_{n,n+1}=\sqrt2 b,
\qquad
(x_n(b)-I)_{n+1,n+2}=\sqrt2 b,
\]
and the additional term \(b^2E_{n,n+2}\) lies on the second superdiagonal.

In a product of upper unitriangular matrices, a term involving two or more strictly upper triangular factors cannot contribute to an entry \((j,j+1)\), because it moves at least two steps to the right. Therefore each first-superdiagonal entry of \(u_2\) is the sum of the corresponding first-superdiagonal entries of the individual factors in
\[
u_2=
 x_1(a_1)\cdots x_{n-1}(a_{n-1})x_n(b)
 x_{n-1}(c_{n-1})\cdots x_1(c_1).
\]
Hence, for \(1\le i\le n-1\),
\[
(u_2)_{i,i+1}=a_i+c_i,
\qquad
(u_2)_{N'-i,N'+1-i}=a_i+c_i,
\]
while
\[
(u_2)_{n,n+1}=\sqrt2 b,
\qquad
(u_2)_{n+1,n+2}=\sqrt2 b.
\]
The matrix formula for the simple-root coordinates gives
\[
\varepsilon_i(u_2)=a_i+c_i\quad(1\le i\le n-1),
\qquad
\varepsilon_n(u_2)=b,
\]
which proves the displayed formulas for the simple-root coordinates.

We prove \eqref{eq:u2-entry-relation}. Since each \(x_r(z)\) is upper unitriangular and has nonzero off-diagonal entries only along the adjacent root positions, except for the term \(b^2E_{n,n+2}\) in \(x_n(b)\), an entry \((u_2)_{i,j}\) with \(j>i\) is obtained by following the increasing chain of indices from \(i\) to \(j\) through the reduced word
\[
1,2,\ldots,n-1,n,n-1,\ldots,1.
\]
For \(2\le i\le n\), the path contributing to the first row reaches column \(i\) through the initial segment \(x_1(a_1)\cdots x_{i-1}(a_{i-1})\), which contributes \(A_{i-1}\). After that, the remaining tail from \(i\) to \(j\) is the same one that computes \((u_2)_{i,j}\). Hence
\[
(u_2)_{1,j}=A_{i-1}(u_2)_{i,j}
\qquad(2\le i\le n,\ j>i).
\]
For the middle row \(i=n+1\), the first row reaches column \(n+2\) through the \(b^2E_{n,n+2}\) term of \(x_n(b)\), after the factor \(A_{n-1}\), whereas row \(n+1\) reaches column \(n+2\) through the adjacent term \(\sqrt2 bE_{n+1,n+2}\). Thus
\[
\sqrt2\,(u_2)_{1,j}
=
A_{n-1}b\,(u_2)_{n+1,j}
\qquad(j>n+1).
\]
Finally, let \(n+2\le i\le 2n\) and put \(\ell=N'-i\). Then the first row reaches column \(i+1=N'+1-\ell\) with factor \(A_{n-1}b^2C_\ell\), while row \(i\) reaches the same column by the simple-root entry \(a_\ell+c_\ell=s_\ell\). The remaining tail is again common, so
\[
A_{n-1}b^2C_\ell\,(u_2)_{i,j}
=
s_\ell (u_2)_{1,j}
\qquad(j>i).
\]
These three identities give the displayed formula for \(\lambda_i\).
\end{proof}

\begin{lemma}\label{lem:triangular-elimination}
Let \(A=t(q)M^\vee u_2\). There is a unique \(u_1\in U_+^\vee\) such that \(u_1A\in B_-^\vee\). It satisfies
\[
\varepsilon_1(u_1)
=
\frac{q(a_1+c_1)}{a_1\cdots a_{n-1}b^2c_{n-1}\cdots c_1},
\qquad
\varepsilon_i(u_1)=0\quad(2\le i\le n).
\]
\end{lemma}

For \(n=2\), with the notation \(a=a_1\) and \(c=c_1\) used above, the corresponding upper unitriangular matrix \(u_1\) is
\[
u_1=
\begin{pmatrix}
1&\dfrac{q(a+c)}{ab^2c}&-\dfrac{\sqrt2q}{ab}&\dfrac{q}{a}&\dfrac{q^2}{ab^2c}\\[1.1em]
0&1&0&0&\dfrac{q}{a}\\[1.1em]
0&0&1&0&\dfrac{\sqrt2q}{ab}\\[1.1em]
0&0&0&1&\dfrac{q(a+c)}{ab^2c}\\[1.1em]
0&0&0&0&1
\end{pmatrix}.
\]

\begin{proof}
We give the calculation explicitly, using the notation \(A_r,C_r,D,s_r\) introduced before Lemma~\ref{lem:u2-contribution}.
The first row of \(u_2\) is
\[
(u_2)_{1,j}=
\begin{cases}
A_{j-2}s_{j-1}, & 2\le j\le n,\\[0.4em]
\sqrt2A_{n-1}b, & j=n+1,\\[0.4em]
A_{n-1}b^2C_{N'+1-j}, & n+2\le j\le N'.
\end{cases}
\]
Moreover the matrix \(A=t(q)M^\vee u_2\) has rows
\[
A_1=q e_{N'}^t,
\qquad
A_i=-(u_2)_i\quad(2\le i\le N'-1),
\qquad
A_{N'}=q^{-1}(u_2)_1,
\]
where \((u_2)_i\) denotes the \(i\)-th row of \(u_2\).

We claim that the required \(u_1\) is the following upper unitriangular matrix. Its nonzero entries off the diagonal are the entries in the last column
\[
(u_1)_{i,N'}=
\begin{cases}
\displaystyle \frac{q}{A_{i-1}}, & 2\le i\le n,\\[1.1em]
\displaystyle \frac{\sqrt2q}{A_{n-1}b}, & i=n+1,\\[1.1em]
\displaystyle \frac{q s_{N'-i}}{A_{n-1}b^2C_{N'-i}}, & n+2\le i\le 2n,
\end{cases}
\]
together with the entries in the first row
\[
(u_1)_{1,j}=
\begin{cases}
\displaystyle (-1)^j\frac{q s_{j-1}}{A_{n-1}b^2C_{j-1}}, & 2\le j\le n,\\[1.1em]
\displaystyle (-1)^{n-1}\frac{\sqrt2q}{A_{n-1}b}, & j=n+1,\\[1.1em]
\displaystyle (-1)^j\frac{q}{A_{N'-j}}, & n+2\le j\le 2n,\\[1.1em]
\displaystyle \frac{q^2}{D}, & j=N'.
\end{cases}
\]
All other off-diagonal entries of \(u_1\) are zero. We verify this claim in two steps.

First, \(u_1\in U_+^\vee\). The displayed matrix is upper unitriangular, so it remains to check that it preserves the form \(J^\vee\). Let
\[
r_j=(u_1)_{1,j}\quad(2\le j\le N'),
\qquad
h_i=(u_1)_{i,N'}\quad(2\le i\le N'-1).
\]
Here \(r\) records the first row and \(h\) records the last column. Write \(C_j\) for the columns of \(u_1\). Since \(u_1\) has hook shape, these columns are
\[
C_1=e_1,\qquad
C_j=e_j+r_je_1\quad(2\le j\le N'-1),
\]
and
\[
C_{N'}=e_{N'}+r_{N'}e_1+\sum_{i=2}^{N'-1}h_ie_i.
\]
Let \(\langle v,w\rangle=v^tJ^\vee w\). Since \(J^\vee\) pairs \(e_i\) only with \(e_{N'+1-i}\), all pairings between columns not involving \(C_{N'}\) are unchanged. Also \(\langle C_1,C_{N'}\rangle=1=\langle e_1,e_{N'}\rangle\). Since \(J^\vee\) is symmetric, the reversed pairings give the same conditions. Hence \(u_1^tJ^\vee u_1=J^\vee\) is equivalent to checking
\[
\langle C_j,C_{N'}\rangle=0\quad(2\le j\le N'-1),
\qquad
\langle C_{N'},C_{N'}\rangle=0.
\]
For \(2\le j\le N'-1\),
\[
\langle C_j,C_{N'}\rangle
=r_j\langle e_1,e_{N'}\rangle
+h_{N'+1-j}\langle e_j,e_{N'+1-j}\rangle
=r_j+(-1)^{j-1}h_{N'+1-j}.
\]
Moreover
\[
\langle C_{N'},C_{N'}\rangle
=2r_{N'}+\sum_{i=2}^{N'-1}(-1)^{i-1}h_ih_{N'+1-i}.
\]
The form condition is therefore equivalent to the two scalar identities
\[
r_j+(-1)^{j-1}h_{N'+1-j}=0
\qquad(2\le j\le N'-1),
\]
and
\[
2r_{N'}+\sum_{i=2}^{N'-1}(-1)^{i-1}h_ih_{N'+1-i}=0.
\]
The first identity follows by direct substitution from the displayed formulae for \(r_j\) and \(h_{N'+1-j}\). For the second, pair the terms \(i\) and \(N'+1-i\), keeping the middle term \(i=n+1\). This gives
\[
\sum_{i=2}^{N'-1}(-1)^{i-1}h_ih_{N'+1-i}
=
\frac{2q^2}{A_{n-1}b^2}
\left(
\sum_{\ell=1}^{n-1}(-1)^\ell\frac{s_\ell}{A_\ell C_\ell}
+\frac{(-1)^n}{A_{n-1}}
\right).
\]
The splitting
\[
\frac{s_\ell}{A_\ell C_\ell}
=
\frac{1}{A_{\ell-1}C_\ell}
+
\frac{1}{A_\ell C_{\ell+1}}
\]
follows from \(s_\ell=a_\ell+c_\ell\), \(A_\ell=A_{\ell-1}a_\ell\), and \(C_\ell=c_\ell C_{\ell+1}\). The alternating sum in parentheses is
\[
\begin{aligned}
&\sum_{\ell=1}^{n-1}(-1)^\ell\frac{s_\ell}{A_\ell C_\ell}
+\frac{(-1)^n}{A_{n-1}}\\
&=
\sum_{\ell=1}^{n-1}(-1)^\ell
\left(
\frac{1}{A_{\ell-1}C_\ell}
+
\frac{1}{A_\ell C_{\ell+1}}
\right)
+\frac{(-1)^n}{A_{n-1}}\\
&=
-\frac{1}{C_1}
+\sum_{\ell=1}^{n-2}
\left(
\frac{(-1)^\ell}{A_\ell C_{\ell+1}}
+
\frac{(-1)^{\ell+1}}{A_\ell C_{\ell+1}}
\right)
+\frac{(-1)^{n-1}+(-1)^n}{A_{n-1}}\\
&=-\frac{1}{C_1},
\end{aligned}
\]
where \(A_0=C_n=1\). Hence
\[
\sum_{i=2}^{N'-1}(-1)^{i-1}h_ih_{N'+1-i}
=
-\frac{2q^2}{A_{n-1}b^2C_1}
=-\frac{2q^2}{D}.
\]
Since \(r_{N'}=q^2/D\), we get
\[
2r_{N'}+\sum_{i=2}^{N'-1}(-1)^{i-1}h_ih_{N'+1-i}=0.
\]
Therefore \(u_1^tJ^\vee u_1=J^\vee\), so \(u_1\in U_+^\vee\).

Second, \(u_1A\in B_-^\vee\). The rows of \(A=t(q)M^\vee u_2\) are
\[
A_1=qe_{N'}^t,\qquad A_i=-(u_2)_i\ (2\le i\le N'-1),\qquad A_{N'}=q^{-1}(u_2)_1.
\]
By \eqref{eq:u2-entry-relation}, and by the displayed definition of the last-column entries \(h_i=(u_1)_{i,N'}=q\lambda_i\), we have
\[
(u_2)_{i,j}=h_iq^{-1}(u_2)_{1,j}
\qquad(2\le i\le N'-1,\ j>i).
\]
For \(2\le i\le N'-1\), the \(i\)-th row of \(u_1A\) is
\[
-(u_2)_i+h_iq^{-1}(u_2)_1.
\]
Equation~\eqref{eq:u2-entry-relation} therefore proves the required vanishing above the diagonal in rows \(2,\ldots,N'-1\).

It remains to check the first row. Since
\[
(u_1)_{1,1}=1,\qquad (u_1)_{1,i}=r_i\quad(2\le i\le N'),
\]
the first row of \(u_1A\) is the following linear combination of the rows of \(A\):
\[
(u_1A)_1
=
A_1+\sum_{i=2}^{N'-1}r_iA_i+r_{N'}A_{N'}
=
qe_{N'}^t-\sum_{i=2}^{N'-1}r_i(u_2)_i
+r_{N'}q^{-1}(u_2)_1.
\]
For \(2\le j<N'\), taking the \(j\)-th component gives
\[
(u_1A)_{1,j}
=
-\sum_{i=2}^{N'-1}r_i(u_2)_{i,j}
+r_{N'}q^{-1}(u_2)_{1,j}.
\]
Since \(u_2\) is upper unitriangular, the terms with \(i>j\) vanish and the term with \(i=j\) is \(-r_j\). For the remaining terms \(2\le i<j\), \eqref{eq:u2-entry-relation} gives
\[
(u_2)_{i,j}=h_iq^{-1}(u_2)_{1,j}.
\]
Hence, for \(2\le j<N'\),
\[
(u_1A)_{1,j}
=
-r_j+\frac{(u_2)_{1,j}}{q}
\left(
r_{N'}-\sum_{i=2}^{j-1}r_ih_i
\right),
\]
where the sum is empty when \(j=2\). For the last column, the row \(A_1=qe_{N'}^t\) contributes \(q\), and \((u_2)_{1,N'}=D\). Therefore
\[
(u_1A)_{1,N'}
=
q-\sum_{i=2}^{N'-1}r_i(u_2)_{i,N'}
+r_{N'}q^{-1}D
=
q+\frac{D}{q}
\left(
r_{N'}-\sum_{i=2}^{N'-1}r_ih_i
\right).
\]
Define
\[
R_j:=r_{N'}-\sum_{i=2}^{j-1}r_ih_i
\qquad(2\le j\le N'),
\]
where the sum is empty for \(j=2\). We compute these quantities from the recursion
\[
R_{j+1}=R_j-r_jh_j.
\]
First,
\[
R_2=r_{N'}=\frac{q^2}{D}
=\frac{q^2}{A_{n-1}b^2A_0C_1}.
\]
Assume \(2\le j\le n\) and that the displayed formula holds for \(R_j\). Using
\[
r_j=(-1)^j\frac{q s_{j-1}}{A_{n-1}b^2C_{j-1}},
\qquad
h_j=\frac{q}{A_{j-1}},
\]
we get
\[
\begin{aligned}
R_{j+1}
&=
(-1)^j\frac{q^2}{A_{n-1}b^2A_{j-2}C_{j-1}}
-
(-1)^j\frac{q^2s_{j-1}}{A_{n-1}b^2A_{j-1}C_{j-1}}\\
&=
(-1)^j\frac{q^2}{A_{n-1}b^2}
\left(
\frac{1}{A_{j-2}C_{j-1}}
-
\frac{s_{j-1}}{A_{j-1}C_{j-1}}
\right)\\
&=
(-1)^{j+1}
\frac{q^2}{A_{n-1}b^2A_{j-1}C_j}.
\end{aligned}
\]
Here the last equality uses \(s_{j-1}=a_{j-1}+c_{j-1}\), \(A_{j-1}=A_{j-2}a_{j-1}\), and \(C_{j-1}=c_{j-1}C_j\). Hence
\[
R_j
=
\frac{(-1)^jq^2}{A_{n-1}b^2A_{j-2}C_{j-1}}
\qquad(2\le j\le n+1),
\]
where \(A_0=C_n=1\). This proves the first range by induction. The middle step gives the initial value for the right-hand tail:
\[
R_{n+2}
=
R_{n+1}-r_{n+1}h_{n+1}
=
\frac{(-1)^{n+1}q^2}{A_{n-1}^2b^2}
-
(-1)^{n-1}\frac{2q^2}{A_{n-1}^2b^2}
=
\frac{(-1)^{n+2}q^2}{A_{n-1}^2b^2}.
\]
This is the formula at \(j=n+2\). Let \(n+2\le j\le 2n-1\), put \(m=N'-j\), and assume the right-hand tail formula holds for \(R_j\). The corresponding formulae for \(r_j\) and \(h_j\) are
\[
r_j=(-1)^j\frac{q}{A_m},
\qquad
h_j=\frac{q s_m}{A_{n-1}b^2C_m}.
\]
Then
\[
\begin{aligned}
R_{j+1}
&=
(-1)^j
\frac{q^2}{A_{n-1}b^2A_mC_{m+1}}
-
(-1)^j
\frac{q^2s_m}{A_{n-1}b^2A_mC_m}\\
&=
(-1)^j\frac{q^2}{A_{n-1}b^2}
\left(
\frac{1}{A_mC_{m+1}}
-
\frac{s_m}{A_mC_m}
\right)\\
&=
(-1)^{j+1}
\frac{q^2}{A_{n-1}b^2A_{m-1}C_m}.
\end{aligned}
\]
Since \(m-1=N'-(j+1)\) and \(C_m=C_{N'+1-(j+1)}\), this proves
\[
R_j
=
\frac{(-1)^jq^2}{A_{n-1}b^2A_{N'-j}C_{N'+1-j}}
\qquad(n+2\le j\le 2n),
\]
Finally,
\[
R_{N'}
=R_{2n}-r_{2n}h_{2n}
=
\frac{q^2}{A_{n-1}b^2A_1C_2}
-
\frac{q^2s_1}{A_{n-1}b^2A_1C_1}
=-\frac{q^2}{A_{n-1}b^2C_1}
=-\frac{q^2}{D}.
\]
Plugging these expressions for \(R_j\) into the displayed formulae for \((u_1A)_{1,j}\) gives, respectively,
\[
(u_1A)_{1,j}
=
-(-1)^j\frac{q s_{j-1}}{A_{n-1}b^2C_{j-1}}
+
\frac{A_{j-2}s_{j-1}}{q}
\frac{(-1)^jq^2}{A_{n-1}b^2A_{j-2}C_{j-1}}
=0
\quad(2\le j\le n),
\]
\[
(u_1A)_{1,n+1}
=
-(-1)^{n-1}\frac{\sqrt2q}{A_{n-1}b}
+
\frac{\sqrt2A_{n-1}b}{q}
\frac{(-1)^{n+1}q^2}{A_{n-1}^2b^2}
=0,
\]
\[
(u_1A)_{1,j}
=
-(-1)^j\frac{q}{A_{N'-j}}
+
\frac{A_{n-1}b^2C_{N'+1-j}}{q}
\frac{(-1)^jq^2}{A_{n-1}b^2A_{N'-j}C_{N'+1-j}}
=0
\quad(n+2\le j\le 2n),
\]
and
\[
(u_1A)_{1,N'}=q+\frac{D}{q}\left(-\frac{q^2}{D}\right)=0.
\]
Thus \((u_1A)_{1,j}=0\) for \(2\le j\le N'\). Consequently
\[
u_1A\in B_-^\vee.
\]
Since bottom-up Gaussian elimination with upper unitriangular left factor is unique whenever it exists, this constructed matrix is the required \(u_1\).

The simple-root coordinates of \(u_1\) are now read off from the adjacent entries:
\[
(u_1)_{1,2}=(u_1)_{N'-1,N'}=\frac{q s_1}{D}
=
\frac{q(a_1+c_1)}{a_1\cdots a_{n-1}b^2c_{n-1}\cdots c_1},
\]
whereas
\[
(u_1)_{i,i+1}=0\quad(2\le i\le n),
\]
including the short-root entries \((u_1)_{n,n+1}=(u_1)_{n+1,n+2}=0\). Hence the stated formulae for \(\varepsilon_i(u_1)\) follow.
\end{proof}

Combining Lemmas~\ref{lem:u2-contribution} and~\ref{lem:triangular-elimination} with
\[
W_{\mathrm{Lie}}(u_1t(q)M^\vee u_2)
=
\sum_{i=1}^n\varepsilon_i(u_1)
+
\sum_{i=1}^n\varepsilon_i(u_2)
\]
gives the superpotential expression directly.

\begin{thm}\label{thm:lusztig-potential}
On the dual-side Lusztig torus for the Rietsch mirror of \(Sp_{2n}/P_1\simeq\PP^{2n-1}\), the superpotential is
\[
W_{\mathrm{Lus}}^{(2n-1)}
=a_1+\cdots+a_{n-1}+b+c_{n-1}+\cdots+c_1
+\frac{q(a_1+c_1)}{a_1\cdots a_{n-1}b^2c_{n-1}\cdots c_1}.
\]
\end{thm}

\section{Toric Degenerations}
This section compares the toric degeneration from the usual toric boundary with the degeneration selected by the type \(C_n\) flag valuation. In both cases, the degree-one value polytope is the polar dual of the Newton polytope of the corresponding Laurent superpotential.

We use the following convention, compatible with the usual polar-dual convention in toric mirror symmetry \cite{Bat94}. If \(P=\operatorname{Newt}(W)\) is the Newton polytope of a Laurent polynomial and the origin is in its interior, then
\[
P^\vee=\{u:\langle u,v\rangle\ge -1\ \text{for all }v\in P\}.
\]
With this convention, in the examples below the degree-one value set of the valuation is \(P^\vee\).

\subsection{Toric degenerations from valuations}
We recall the valuation mechanism behind the toric degenerations used below. Let \(X\) be a projective variety with an ample line bundle \(L\), and let
\[
R=\bigoplus_{k\ge0}H^0(X,L^k)
\]
be its section ring. A valuation on \(R\), ordered lexicographically on each graded piece, must be combined with the grading to produce the value semigroup. Thus, for \(0\ne s\in H^0(X,L^k)\), we use the degree-augmented value
\[
\widetilde\nu(s)=(k,\nu(s)).
\]
The corresponding value semigroup is
\[
S(R,\nu)=\{(k,\nu(s)):0\ne s\in H^0(X,L^k)\}
\]
in \(\ZZ_{\ge0}\times\ZZ^d\). If the valuation has one-dimensional leaves and this semigroup is finitely generated, then the associated graded algebra is
\[
\operatorname{gr}_\nu R\simeq \CC[S(R,\nu)].
\]
The Rees algebra of the valuation filtration then gives a flat degeneration
\[
X=\operatorname{Proj}R\rightsquigarrow \operatorname{Proj}\operatorname{gr}_\nu R.
\]
The special fibre is therefore the projective toric variety determined by the cone over the Newton--Okounkov body. This is the form of Anderson's Newton--Okounkov toric degeneration theorem used below \cite{Anderson}; see also the foundational constructions of Newton--Okounkov bodies in \cite{KK,LM} and the integrable-systems perspective in \cite{HK}.

A standard way to obtain such a valuation is to choose an admissible flag
\[
X=Y_0\supset Y_1\supset\cdots\supset Y_d=\{p\},
\]
with each \(Y_i\) irreducible and smooth at \(p\). The flag valuation records successive orders of vanishing along \(Y_1,Y_2,\ldots,Y_d\). We first illustrate this in the explicit \(\PP^3\) calculation, then record the standard toric valuation for \(\PP^N\), and finally use the type \(C_n\) flag whose value polytope is the polar dual of the Newton polytope of the Lie-theoretical Laurent polynomial.

\subsection{Toy example: \(\PP^3\)}
Let \(X=\PP^3\) with homogeneous coordinates \([p_0:p_1:p_2:p_3]\). The Lie-theoretical calculation in this example is based on Tillmann-Morris's unpublished mini-project report \cite{TillmannMorris}; Section~6 of the same manuscript also records the odd-dimensional formula used in Theorem~\ref{thm:lusztig-potential}. For the standard toric boundary \(p_0p_1p_2p_3=0\), work on the affine chart \(p_0\ne0\) and write
\[
x=\frac{p_1}{p_0},
\qquad
y=\frac{p_2}{p_0},
\qquad
z=\frac{p_3}{p_0}.
\]
The toric mirror potential (see Example~\ref{ex:tor-PN}) is
\[
W_{xyz}=x+y+z+\frac{1}{xyz}.
\]
Its Newton polytope is
\[
P_{xyz}
=
\operatorname{Conv}\{(1,0,0),(0,1,0),(0,0,1),(-1,-1,-1)\}.
\]
Its polar dual is
\[
P_{xyz}^\vee
=
\{(X,Y,Z)\in\RR^3:
X,Y,Z\ge -1,\ X+Y+Z\le1\}.
\]
This is the usual valuation polytope for the toric degeneration of \(\PP^3\).

On the other hand, the type \(C_2\) Lie-theoretical calculation of Theorem~\ref{thm:lusztig-potential} gives, after setting the quantum parameter \(q=1\) for the purpose of taking the Newton polytope, the Laurent polynomial
\[
W_{abc}
=
a+b+c+\frac{1}{ab^2}+\frac{1}{b^2c}.
\]
The support exponents are
\[
(1,0,0),\quad (0,1,0),\quad (0,0,1),
\quad (-1,-2,0),\quad (0,-2,-1).
\]
Hence, if \(P_{abc}=\operatorname{Newt}(W_{abc})\), then
\[
P_{abc}^\vee
=
\{(A,B,C)\in\RR^3:
A,B,C\ge -1,\ A+2B\le1,\ 2B+C\le1\}.
\]

\begin{rmk}
When \(n=2\), the Laurent polynomial obtained from the type \(C_2\) Lusztig torus is mutation-equivalent to the standard toric Hori--Vafa potential. In the sense of Akhtar--Coates--Galkin--Kasprzyk~\cite{ACGK}, the following birational map is the composition of three Laurent mutations:
\[
a=\frac{xy}{x+z},\qquad
b=x+z,\qquad
c=\frac{yz}{x+z}.
\]
It sends \(W_{xyz}=x+y+z+1/(xyz)\) to \(W_{abc}=a+b+c+1/(ab^2)+1/(b^2c)\). This comparison is included only as motivation and as a link with the MMLP picture for Fano threefolds \cite{CKPT}.
\end{rmk}

The two polytopes in this toy example are shown in Figure~\ref{fig:p3-polytopes}.

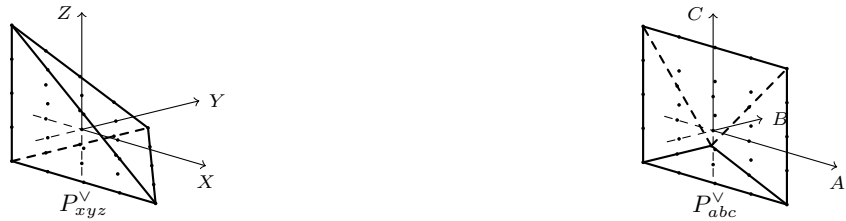
\begin{figure}[ht]
\centering
\begin{minipage}{0.47\textwidth}
\centering
\begin{tikzpicture}[scale=.5,x={(0.95cm,-0.28cm)},y={(0.9cm,0.22cm)},z={(0cm,0.9cm)},line join=round,line cap=round]
\coordinate (O) at (-1,-1,-1);
\coordinate (X) at (3,-1,-1);
\coordinate (Y) at (-1,3,-1);
\coordinate (Z) at (-1,-1,3);
\draw[dashed] (-1.35,0,0)--(0,0,0);
\draw[dashed] (0,-1.35,0)--(0,0,0);
\draw[dashed] (0,0,-1.35)--(0,0,0);
\draw[->] (0,0,0)--(3.45,0,0) node[below] {\scriptsize \(X\)};
\draw[->] (0,0,0)--(0,3.45,0) node[right] {\scriptsize \(Y\)};
\draw[->] (0,0,0)--(0,0,3.45) node[left] {\scriptsize \(Z\)};
\foreach \Xcoord in {-1,0,1,2,3}{
  \foreach \Ycoord in {-1,0,1,2,3}{
    \foreach \Zcoord in {-1,0,1,2,3}{
      \pgfmathtruncatemacro{\scoord}{\Xcoord+\Ycoord+\Zcoord}
      \ifnum\scoord<2
        \fill (\Xcoord,\Ycoord,\Zcoord) circle (1.35pt);
      \fi
    }
  }
}
\draw[thick,dashed] (O)--(Y);
\draw[thick] (O)--(X)--(Y)--(Z)--(O);
\draw[thick] (X)--(Z);
\node[below] at (1,-1,-1) {\small \(P_{xyz}^\vee\)};
\end{tikzpicture}
\end{minipage}
\hfill
\begin{minipage}{0.47\textwidth}
\centering
\begin{tikzpicture}[scale=.5,x={(0.95cm,-0.28cm)},y={(0.9cm,0.22cm)},z={(0cm,0.9cm)},line join=round,line cap=round]
\coordinate (O) at (-1,-1,-1);
\coordinate (X) at (3,-1,-1);
\coordinate (Z) at (-1,-1,3);
\coordinate (T) at (-1,1,-1);
\coordinate (U) at (3,-1,3);
\draw[dashed] (-1.35,0,0)--(0,0,0);
\draw[dashed] (0,-1.35,0)--(0,0,0);
\draw[dashed] (0,0,-1.35)--(0,0,0);
\draw[->] (0,0,0)--(3.45,0,0) node[below] {\scriptsize \(A\)};
\draw[->] (0,0,0)--(0,1.45,0) node[right] {\scriptsize \(B\)};
\draw[->] (0,0,0)--(0,0,3.45) node[left] {\scriptsize \(C\)};
\foreach \Acoord in {-1,0,1,2,3}{
  \foreach \Bcoord in {-1,0,1}{
    \foreach \Ccoord in {-1,0,1,2,3}{
      \pgfmathtruncatemacro{\sAcoord}{\Acoord+2*\Bcoord}
      \pgfmathtruncatemacro{\sCcoord}{\Ccoord+2*\Bcoord}
      \ifnum\sAcoord<2
        \ifnum\sCcoord<2
          \fill (\Acoord,\Bcoord,\Ccoord) circle (1.35pt);
        \fi
      \fi
    }
  }
}
\draw[thick,dashed] (T)--(Z);
\draw[thick,dashed] (T)--(U);
\draw[thick] (O)--(X)--(U)--(Z)--(O);
\draw[thick] (O)--(T)--(X);
\node[below] at (1,-1,-1) {\small \(P_{abc}^\vee\)};
\end{tikzpicture}
\end{minipage}
\caption{The polar duals of the Newton polytopes of \(W_{xyz}\) and \(W_{abc}\). The black dots are the lattice points in the degree-one normalized value sets after the indicated lattice coordinate identifications.}
\label{fig:p3-polytopes}
\end{figure}

Let
\[
\phi=p_0p_3-p_1p_2,
\qquad
Q=\{\phi=0\}\subset\PP^3.
\]
Then \(Q\simeq\PP^1\times\PP^1\). Use the Segre parametrization
\[
[p_0:p_1:p_2:p_3]
=
[s_0t_0:s_0t_1:s_1t_0:s_1t_1],
\]
and choose the flag
\[
\PP^3\supset Q\supset \mathcal C\supset p_\infty,
\]
where
\[
\mathcal C=\{s_0=0\}=\{p_0=p_1=0\}\subset Q,
\qquad
p_\infty=[0:0:0:1].
\]
Put
\[
s_\star=p_0p_3\phi=p_0p_3(p_0p_3-p_1p_2).
\]
For \(s\in H^0(\PP^3,\mathcal O(4k))\), write \(r=\operatorname{ord}_Q(s)\), remove the factor \(\phi^r\), and restrict the residual section to \(Q\). If \((i,j)\) is the standard toric flag value on \(Q\simeq\PP^1\times\PP^1\), the original flag valuation is
\[
\nu(s)=(r,i,j).
\]
Let
\[
\sigma:\ZZ^3\longrightarrow\ZZ^3,
\qquad
\sigma(r,i,j)=(i,r,j).
\]
For \(0\ne s\in H^0(\PP^3,\mathcal O(4k))\), define the normalized lattice-labelled value
\[
\lambda_{abc,k}(s)
:=
\sigma\bigl(\nu(s)-k\mathbf 1\bigr),
\qquad
\mathbf 1=(1,1,1).
\]
Thus
\[
\lambda_{abc,k}(s)
=
(i-k,\ r-k,\ j-k).
\]
The map \(\sigma\) is used only to identify the flag-value lattice with the exponent lattice of \(W_{abc}\). We do not equip the reordered coordinates with the standard lexicographic order and regard the resulting map as a new valuation.
The degree \(k\) normalized value set is
\[
\left\{
\lambda_{abc,k}(s):
0\ne s\in H^0(\PP^3,\mathcal O(4k))
\right\}
=
kP_{abc}^{\vee}\cap\ZZ^3.
\]
To see this directly, use a basis adapted to the filtration by powers of \(\phi\). For \(0\le r\le 2k\), put \(d=4k-2r\). Restriction identifies the \(r\)-th associated graded quotient with \(H^0(Q,\mathcal O_Q(d,d))\). On \(Q\simeq\PP^1\times\PP^1\), the monomials \(s_0^is_1^{d-i}t_0^jt_1^{d-j}\), \(0\le i,j\le d\), have toric flag values \((i,j)\). After choosing homogeneous lifts, the adapted basis elements have normalized values \((i-k,r-k,j-k)\), where \(0\le r\le2k\) and \(0\le i,j\le4k-2r\). Thus the degree \(k\) value set is cut out by \(A,B,C\ge-k\), \(A+2B\le k\), and \(2B+C\le k\). Therefore
\[
\left\{
\lambda_{abc,k}(s):
0\ne s\in H^0(\PP^3,\mathcal O(4k))
\right\}
=
kP_{abc}^\vee\cap\ZZ^3.
\]

\begin{exam}
Take \(s=\phi\,p_0^2\). Then \(k=1\), \(r=1\), and \(p_0^2|_Q=(s_0t_0)^2\), so the toric flag value on \(Q\simeq\PP^1\times\PP^1\) is \((i,j)=(2,2)\). Hence \(\lambda_{abc,1}(s)=(1,0,1)\), which lies on the two facets \(A+2B=1\) and \(2B+C=1\) of \(P_{abc}^\vee\). This agrees with the case \(n=2\) of the general construction below, since \(\phi=-F_1\).
\end{exam}

\subsection{The standard toric valuation}
We now record the analogous valuation for the classical toric mirror of \(\PP^N\). Let \(H_{\mathrm{tor}}=p_0p_1\cdots p_N\) be the toric anticanonical section. The anticanonical section ring is \(R_{\mathrm{tor}}=\bigoplus_{k\ge0}H^0(\PP^N,\mathcal O((N+1)k))\). On the affine chart \(p_0\ne0\), set \(x_i=p_i/p_0\) for \(1\le i\le N\).
After dividing a degree \((N+1)k\) section by \(H_{\mathrm{tor}}^k\), it becomes a Laurent polynomial in the \(x_i\)'s. We use the monomial valuation
\[
\nu_{\mathrm{tor}}\left(\sum_m c_mx^m\right)
=
\min_{\mathrm{lex}}\{m\in\ZZ^N:c_m\ne0\},
\]
so that \(\nu_{\mathrm{tor}}(x^m)=m\). This is the valuation obtained by recording the exponent vector in the standard torus coordinates.

Let
\[
P_{\mathrm{tor}}^{(N)}
=
\operatorname{Newt}\left(x_1+\cdots+x_N+\frac{q}{x_1\cdots x_N}\right)
=
\operatorname{Conv}\{e_1,\ldots,e_N,-\mathbf 1\},
\]
where \(\mathbf 1=(1,\ldots,1)\).

\begin{thm}\label{thm:standard-toric-valuation}
For every \(k\ge0\),
\[
\left\{
\nu_{\mathrm{tor}}\left(\frac{s}{H_{\mathrm{tor}}^k}\right):
0\ne s\in H^0(\PP^N,\mathcal O((N+1)k))
\right\}
=
k\bigl(P_{\mathrm{tor}}^{(N)}\bigr)^\vee\cap\ZZ^N.
\]
If \(\widetilde\nu_{\mathrm{tor}}(s)=(k,\nu_{\mathrm{tor}}(s/H_{\mathrm{tor}}^k))\), then the corresponding Rees algebra gives the standard toric degeneration associated with the polar dual of the Newton polytope of the classical toric mirror.
\end{thm}

\begin{proof}
Use the standard monomial basis of \(H^0(\PP^N,\mathcal O((N+1)k))\):
\[
p^\alpha=p_0^{\alpha_0}p_1^{\alpha_1}\cdots p_N^{\alpha_N},
\qquad
|\alpha|=(N+1)k.
\]
On the chart \(p_0\ne0\), we have
\[
\frac{p^\alpha}{H_{\mathrm{tor}}^k}
=
x_1^{\alpha_1-k}\cdots x_N^{\alpha_N-k}.
\]
Thus its value is
\[
m=(\alpha_1-k,\ldots,\alpha_N-k).
\]
The conditions \(\alpha_i\ge0\) for \(1\le i\le N\) give \(m_i\ge-k\). The remaining condition
\(\alpha_0=(N+1)k-\sum_{i=1}^N\alpha_i\ge0\) gives \(\sum_i m_i\le k\). Conversely, any integral \(m\) satisfying these inequalities is obtained by setting
\[
\alpha_i=m_i+k\quad(1\le i\le N),
\qquad
\alpha_0=k-\sum_{i=1}^Nm_i.
\]
Since the lexicographic monomial valuation of a nonzero linear combination is the value of one of its monomial terms, the degree \(k\) value set is exactly
\[
\left\{
\nu_{\mathrm{tor}}\left(\frac{s}{H_{\mathrm{tor}}^k}\right):
0\ne s\in H^0(\PP^N,\mathcal O((N+1)k))
\right\}
=
\{m\in\ZZ^N:m_i\ge-k,\ \sum_i m_i\le k\}.
\]

On the other hand, \(P_{\mathrm{tor}}^{(N)}=\operatorname{Conv}\{e_1,\ldots,e_N,-\mathbf 1\}\). By the polar convention fixed at the beginning of this section, \(\bigl(P_{\mathrm{tor}}^{(N)}\bigr)^\vee\) is cut out by
\[
u_i\ge -1\qquad(1\le i\le N),
\qquad
\sum_{i=1}^Nu_i\le1.
\]
Thus
\[
k\bigl(P_{\mathrm{tor}}^{(N)}\bigr)^\vee\cap\ZZ^N
=
\{m\in\ZZ^N:m_i\ge-k,\ \sum_i m_i\le k\},
\]
which is exactly the value set computed above.

After adjoining the degree, the value semigroup is
\[
S_{\mathrm{tor}}
=
\{(k,m):k\ge0,\ m\in k(P_{\mathrm{tor}}^{(N)})^\vee\cap\ZZ^N\}.
\]
It is finitely generated by Gordan's lemma, and the monomial valuation has one-dimensional leaves. Anderson's Newton--Okounkov theorem therefore gives a flat Rees degeneration with special fibre \(\operatorname{Proj}\CC[S_{\mathrm{tor}}]\), the projective toric variety associated with \((P_{\mathrm{tor}}^{(N)})^\vee\).
\end{proof}

\subsection{The flag valuation and the Newton--Okounkov body}
Let \(X=\PP^{2n-1}\) with homogeneous coordinates \([p_0:\cdots:p_N]\), where \(N=2n-1\). In the previous section the type \(C_n\) boundary \(D_C\) was defined by \(H_N=p_0p_N\prod_{r=1}^{n-1}F_r\), with \(F_r=\sum_{j=0}^r(-1)^jp_{r-j}p_{N-r+j}\).
The following flag and valuation generalize the type \(C\) calculation. They give the valuation-side construction of the toric degeneration adapted to the type \(C_n\) anticanonical section \(H_N\).
Set \(F_0=p_0p_N\). Then, for \(1\le r\le n-1\), \(F_r=p_rp_{N-r}-F_{r-1}\).
Let \(Q_r=\{F_r=0\}\). The flag used below starts with \(Y_0=\PP^N\) and \(Y_1=Q_{n-1}\). Then, for \(m=n-1,n-2,\ldots,2\), impose successively the two hyperplane conditions \(p_m=0\) and \(p_{N-m}=0\).
After these steps the remaining surface is
\[
Q_1=\{p_1p_{N-1}-p_0p_N=0\}
\subset \PP\langle p_0,p_1,p_{N-1},p_N\rangle\cong \PP^3.
\]
This is a smooth quadric surface. Identify it with \(\PP^1\times\PP^1\) by
\[
p_0=X_0Y_0,
\qquad
p_1=X_1Y_0,
\qquad
p_{N-1}=X_0Y_1,
\qquad
p_N=X_1Y_1.
\]
On \(Q_1\), finish the flag by taking
\[
C=\{Y_0=0\}=\{p_0=p_1=0\}
\]
and
\[
p_\infty=([0:1],[0:1])=[0:\cdots:0:1].
\]
Let \(Y_\bullet\) denote the resulting full flag. Near \(p_\infty\), with local coordinates
\[
x=\frac{X_0}{X_1}=\frac{p_{N-1}}{p_N},
\qquad
 y=\frac{Y_0}{Y_1}=\frac{p_1}{p_N},
\]
the final flag is \(Q_1\supset\{y=0\}\supset\{x=y=0\}\).

\begin{lemma}\label{lem:flag-admissible}
The flag \(Y_\bullet\) is admissible: every member is irreducible and smooth at \(p_\infty\), and \(Y_i\) has codimension \(i\) in \(\PP^N\).
\end{lemma}

\begin{proof}
For \(1\le m\le n-1\), set
\[
P_m:=\PP\langle p_0,\ldots,p_m,p_{N-m},\ldots,p_N\rangle
\]
and
\[
Z_m:=\{F_m=0\}\subset P_m.
\]
The quadratic form \(F_m\) is nondegenerate on the coordinates of \(P_m\), so \(Z_m\) is an irreducible quadric for every \(m\ge1\).
The identity
\[
F_m=p_mp_{N-m}-F_{m-1}
\]
is the key point.  Inside \(Z_m\), the intersection with \(p_m=0\) is cut out in \(P_m\) by
\[
p_m=0,\qquad F_{m-1}=0,
\]
while \(p_{N-m}\) remains a free homogeneous coordinate. Hence \(Z_m\cap\{p_m=0\}\) is the projective cone over \(Z_{m-1}\) with vertex in the \(p_{N-m}\)-direction, and is therefore irreducible. Intersecting further with \(p_{N-m}=0\) gives \(Z_{m-1}\subset P_{m-1}\). Together with the irreducibility of the quadrics \(Z_m\), this proves irreducibility of every stratum in the flag and also gives the expected codimensions.

For smoothness at \(p_\infty=[0:\cdots:0:1]\), work in the affine chart \(p_N=1\), with local coordinates \(p_0,\ldots,p_{N-1}\). With the present indexing,
\[
F_m=(-1)^mp_0+\text{terms of order at least }2
\qquad\text{at }p_\infty.
\]
Thus \(F_m\) may replace \(p_0\) as one element of a regular system of parameters in the local ring at \(p_\infty\). At the stage indexed by \(m\), the local equations are \(F_m\) together with the coordinate functions already imposed, namely \(p_{m+1},p_{N-m-1},\ldots,p_{n-1},p_n\), and then successively \(p_m\) and \(p_{N-m}\). These equations have independent linear parts after replacing \(p_0\) by \(F_m\). At the final surface \(Q_1\), the last two equations are \(p_1=0\) and \(p_{N-1}=0\), equivalently \(y=0\) and \(x=0\) in the coordinates on \(Q_1\simeq\PP^1\times\PP^1\). Hence the equations defining every stratum form part of a regular system of parameters at \(p_\infty\), so every stratum is smooth there.
\end{proof}

Let
\[
R=\bigoplus_{k\ge0}H^0(\PP^N,\mathcal O_{\PP^N}((N+1)k)).
\]
Set \(d=2n-1\), and let \(\mathbf 1=(1,\ldots,1)\in\ZZ^d\).
For \(s\in H^0(\PP^N,\mathcal O((N+1)k))\), let
\[
\nu_{\mathrm{flag}}(s)=(r,u_{n-1},v_{n-1},\ldots,u_2,v_2,u_1,v_1)
\]
be the flag valuation just described. Thus \(r\) is the order along \(Q_{n-1}\); for \(2\le m\le n-1\), the pair \((u_m,v_m)\) records the successive orders along \(p_m=0\) and \(p_{N-m}=0\); and \((u_1,v_1)\) is the toric flag valuation on \(Q_1\simeq\PP^1\times\PP^1\), with \(u_1=\operatorname{ord}_C\) and \(v_1=\operatorname{ord}_{p_\infty}\) after restricting to \(C\).

For
\[
H_N=p_0p_N\prod_{r=1}^{n-1}F_r,
\]
one has
\[
\nu_{\mathrm{flag}}(H_N)=\mathbf 1.
\]
Indeed, \(F_{n-1}\) gives the first order, and the identity \(F_r=p_rp_{N-r}-F_{r-1}\) shows inductively that, after restricting to \(Q_r\), the previous factor \(F_{r-1}\) restricts to \(p_rp_{N-r}\). At the last step, \(p_0p_N\) restricts on \(Q_1\) to the local monomial \(xy\), which has final valuation \((1,1)\).

For \(0\ne s\in H^0(\PP^N,\mathcal O((N+1)k))\), define the raw normalized value vector
\[
\lambda_k(s):=\nu_{\mathrm{flag}}(s)-k\mathbf 1.
\]
Write the shifted coordinates as
\[
B=r-k,
\qquad
A_m=u_m-k,
\qquad
C_m=v_m-k
\quad(1\le m\le n-1).
\]
Thus
\[
\lambda_k(s)
=
(B,A_{n-1},C_{n-1},\ldots,A_1,C_1).
\]

\begin{rmk}\label{rmk:lattice-identification}
The Lusztig monomial coordinates use the opposite order for the \(A_m,C_m\)-pairs. To compare value vectors with those exponent coordinates, let \(\sigma:\ZZ^d\longrightarrow\ZZ^d\) be the lattice permutation \(\sigma(B,A_{n-1},C_{n-1},\ldots,A_1,C_1)=(B,A_1,C_1,\ldots,A_{n-1},C_{n-1})\).
The map \(\sigma\) is only a lattice coordinate identification. We do not regard \(\sigma\circ\lambda_k\), with the standard lexicographic order on the target coordinates after applying \(\sigma\), as a new valuation.

On the degree-augmented lattice define
\[
\Sigma:\ZZ\times\ZZ^d\longrightarrow\ZZ\times\ZZ^d,
\qquad
\Sigma(k,v)=\bigl(k,\sigma(v-k\mathbf 1)\bigr).
\]
For \((k,v),(\ell,w)\in\ZZ\times\ZZ^d\),
\[
\begin{aligned}
\Sigma\bigl((k,v)+(\ell,w)\bigr)
&=
\Sigma(k+\ell,v+w)\\
&=
\bigl(k+\ell,\sigma(v+w-(k+\ell)\mathbf 1)\bigr)\\
&=
\bigl(k+\ell,\sigma(v-k\mathbf 1)+\sigma(w-\ell\mathbf 1)\bigr)\\
&=
\Sigma(k,v)+\Sigma(\ell,w).
\end{aligned}
\]
Thus \(\Sigma\) is a group homomorphism. Its inverse is
\[
\Sigma^{-1}(k,m)=\bigl(k,\sigma^{-1}(m)+k\mathbf 1\bigr),
\]
so \(\Sigma\) is a unimodular lattice automorphism. The toric degeneration below is defined by the original flag valuation; \(\sigma\) and \(\Sigma\) are used only to identify lattice coordinates and value semigroups.
\end{rmk}

With the polar convention fixed above, in the coordinates
\[
(B,A_1,C_1,\ldots,A_{n-1},C_{n-1}),
\]
the polar polytope attached to the type \(C_n\) Lie-theoretical support is
\[
P_n^\vee=\left\{(B,A_1,C_1,\ldots,A_{n-1},C_{n-1})\in\RR^{2n-1}:
\begin{array}{l}
B\ge -1,\ A_m\ge -1,\ C_m\ge -1,\\
A_1+2B+\sum_{m=2}^{n-1}(A_m+C_m)\le1,\\
C_1+2B+\sum_{m=2}^{n-1}(A_m+C_m)\le1
\end{array}
\right\}.
\]
Equivalently, \(P_n^\vee\) is defined by
\[
B\ge -1,
\qquad
A_m\ge -1,
\qquad
C_m\ge -1
\quad(1\le m\le n-1),
\]
and
\[
A_1+2B+\sum_{m=2}^{n-1}(A_m+C_m)\le1,
\]
\[
C_1+2B+\sum_{m=2}^{n-1}(A_m+C_m)\le1.
\]

\begin{prop}\label{prop:no-body-p2n}
For every \(k\ge0\),
\[
\sigma\left(
\left\{
\nu_{\mathrm{flag}}(s)-k\mathbf 1:
0\ne s\in H^0(\PP^N,\mathcal O((N+1)k))
\right\}
\right)
=
kP_n^\vee\cap\ZZ^d.
\]
\end{prop}

\begin{proof}
Let \(s\in H^0(\PP^N,\mathcal O((N+1)k))\), and write
\[
\nu_{\mathrm{flag}}(s)=(r,u_{n-1},v_{n-1},\ldots,u_2,v_2,u_1,v_1).
\]
Write \(B=r-k\), \(A_m=u_m-k\), and \(C_m=v_m-k\).
All entries of \(\nu_{\mathrm{flag}}(s)\) are nonnegative. After removing \(r\) copies of \(F_{n-1}\), the residual degree is \((N+1)k-2r\). After imposing the successive orders \(u_m,v_m\) for \(2\le m\le n-1\), the residual degree on the final quadric \(Q_1\) is
\[
\delta=(N+1)k-2r-
\sum_{m=2}^{n-1}(u_m+v_m).
\]
The restriction of \(\mathcal O_{\PP^N}(\delta)\) to \(Q_1\simeq\PP^1\times\PP^1\) is \(\mathcal O_{Q_1}(\delta,\delta)\). For the standard toric flag on \(\PP^1\times\PP^1\), the value \((u_1,v_1)\) of a section of \(\mathcal O(\delta,\delta)\) satisfies
\[
0\le u_1\le \delta,
\qquad
0\le v_1\le \delta.
\]
Substituting
\[
r=B+k,
\qquad
u_m=A_m+k,
\qquad
v_m=C_m+k
\]
turns these two upper bounds into
\[
A_1+k
\le
(N+1)k-2(B+k)
-
\sum_{m=2}^{n-1}(A_m+C_m+2k).
\]
Since \(N+1=2n\), the constant terms on the right simplify as
\[
(N+1)k-2k-2(n-2)k=2k.
\]
Thus the displayed inequality is equivalent to
\[
A_1+2B+
\sum_{m=2}^{n-1}(A_m+C_m)\le k,
\]
and the same calculation with \(C_1+k\le \delta\) gives
\[
C_1+2B+
\sum_{m=2}^{n-1}(A_m+C_m)\le k.
\]
The nonnegativity of the unshifted valuation entries gives
\[
B\ge -k,
\qquad
A_m\ge -k,
\qquad
C_m\ge -k.
\]
Thus every value vector after applying \(\sigma\) lies in \(kP_n^\vee\cap\ZZ^d\).

Conversely, take any integral point
\[
(B,A_1,C_1,\ldots,A_{n-1},C_{n-1})\in kP_n^\vee\cap\ZZ^d.
\]
Set
\[
r=B+k,
\qquad
u_m=A_m+k,
\qquad
v_m=C_m+k
\quad(1\le m\le n-1).
\]
The lower-bound inequalities imply that all these integers are nonnegative. Define
\[
\delta=(N+1)k-2r-
\sum_{m=2}^{n-1}(u_m+v_m).
\]
The two upper-bound inequalities are precisely
\[
0\le u_1\le \delta,
\qquad
0\le v_1\le \delta.
\]
On \(Q_1\simeq\PP^1\times\PP^1\), choose the section
\[
\sigma_{u_1,v_1}
=
X_0^{v_1}X_1^{\delta-v_1}Y_0^{u_1}Y_1^{\delta-u_1}
\in H^0(Q_1,\mathcal O_{Q_1}(\delta,\delta)).
\]
In the local coordinates \(x=X_0/X_1\), \(y=Y_0/Y_1\) at \(p_\infty\), this section is \(x^{v_1}y^{u_1}\) times a unit, so its toric flag valuation is \((u_1,v_1)\).

Let
\[
L=\PP\langle p_0,p_1,p_{N-1},p_N\rangle\cong\PP^3.
\]
The exact sequence
\[
0\to\mathcal O_L(\delta-2)\to\mathcal O_L(\delta)\to\mathcal O_{Q_1}(\delta,\delta)\to0
\]
and the vanishing \(H^1(L,\mathcal O_L(\delta-2))=0\) for \(\delta\ge0\) show that the restriction map
\[
H^0(L,\mathcal O_L(\delta))\longrightarrow H^0(Q_1,\mathcal O_{Q_1}(\delta,\delta))
\]
is surjective. Choose a homogeneous lift \(\widetilde\sigma\in H^0(L,\mathcal O_L(\delta))\) of \(\sigma_{u_1,v_1}\), and regard it as a homogeneous polynomial in \(p_0,p_1,p_{N-1},p_N\), hence as a section on \(\PP^N\).

Now define
\[
s=
F_{n-1}^r
\left(\prod_{m=2}^{n-1}p_m^{u_m}p_{N-m}^{v_m}\right)
\widetilde\sigma.
\]
Its degree is
\[
2r+
\sum_{m=2}^{n-1}(u_m+v_m)+\delta
=
(N+1)k.
\]
The successive orders are exact. First, \(F_{n-1}\) does not divide the residual factor after removing \(F_{n-1}^r\). Indeed, \(F_{n-1}\) is an irreducible quadratic and is not a coordinate hyperplane factor, so it cannot divide the coordinate monomial part. For \(n>2\), it also cannot divide \(\widetilde\sigma\), since \(\widetilde\sigma\) only uses the final variables \(p_0,p_1,p_{N-1},p_N\), whereas \(F_{n-1}\) contains the term \(p_{n-1}p_n\). For \(n=2\), divisibility by \(F_1\) would force \(\widetilde\sigma|_{Q_1}=0\), contradicting the choice of the nonzero section \(\sigma_{u_1,v_1}\). Hence the order along \(Q_{n-1}\) is \(r\).

For the intermediate orders, argue recursively. Suppose the steps with indices larger than \(m\) have already been passed. The current stratum is \(Z_m=\{F_m=0\}\subset P_m\), and the remaining section has the form
\[
p_m^{u_m}p_{N-m}^{v_m}
\left(\prod_{\ell=2}^{m-1}p_\ell^{u_\ell}p_{N-\ell}^{v_\ell}\right)
\widetilde\sigma
\]
up to a unit along the generic point of the current stratum. On \(Z_m\), the next member of the flag is the divisor \(p_m=0\). Hence \(p_m\) is the local equation whose order is measured at this step. The other displayed factors are not identically zero along the generic point of \(Z_m\cap\{p_m=0\}\): in particular, \(p_{N-m}\) remains free there, and the lower-index coordinates have not yet been imposed. Thus the order along \(p_m=0\) is \(u_m\).

After dividing by \(p_m^{u_m}\) and restricting to \(Z_m\cap\{p_m=0\}\), the identity
\[
F_m=p_mp_{N-m}-F_{m-1}
\]
shows that this intermediate stratum is the projective cone over \(Z_{m-1}\), with \(p_{N-m}\) a Cartier coordinate on the cone direction. The next member of the flag is obtained by imposing \(p_{N-m}=0\). Hence \(p_{N-m}^{v_m}\) contributes order \(v_m\), and after dividing by this power and restricting to \(p_{N-m}=0\), the current stratum becomes \(Z_{m-1}\). Repeating this argument for \(m=n-1,n-2,\ldots,2\) gives all intermediate orders. The lift \(\widetilde\sigma\) uses only the final variables \(p_0,p_1,p_{N-1},p_N\), so it introduces no additional vanishing along these intermediate branches. Finally, \(\sigma_{u_1,v_1}\) was chosen to have toric flag value \((u_1,v_1)\) on \(Q_1\simeq\PP^1\times\PP^1\).

It follows that the successive restrictions along the flag have orders
\[
(r,u_{n-1},v_{n-1},\ldots,u_2,v_2,u_1,v_1).
\]
Thus
\[
\sigma\bigl(\nu_{\mathrm{flag}}(s)-k\mathbf 1\bigr)
=
(B,A_1,C_1,\ldots,A_{n-1},C_{n-1}),
\]
which is the prescribed lattice point after the lattice identification \(\sigma\). Hence every lattice point of \(kP_n^\vee\) occurs.
\end{proof}

\begin{thm}\label{thm:no-toric-degeneration}
Let
\[
S_{\mathrm{flag}}
=
\left\{
\bigl(k,\nu_{\mathrm{flag}}(s)\bigr):
k\ge0,\ 0\ne s\in R_k
\right\},
\qquad
R_k=H^0(\PP^N,\mathcal O((N+1)k)).
\]
Let
\[
S_n=\{(k,m):k\ge0,\ m\in kP_n^\vee\cap\ZZ^d\}.
\]
Then
\[
\Sigma(S_{\mathrm{flag}})=S_n.
\]
Moreover,\footnote{Here and below, \(\operatorname{gr}_{\nu_{\mathrm{flag}}}R\) denotes the associated graded algebra for the filtration on the graded ring \(R\) induced by the degree-augmented value \(0\ne s\in R_k\mapsto(k,\nu_{\mathrm{flag}}(s))\). We retain the shorter notation because the degree is already recorded by the grading of \(R\).}
\[
\operatorname{gr}_{\nu_{\mathrm{flag}}}R
\simeq \CC[S_{\mathrm{flag}}]\simeq \CC[S_n].
\]
The Rees algebra of the filtration defined by the original flag valuation gives a flat degeneration
\[
\PP^N\rightsquigarrow X_{P_n^\vee}:=\operatorname{Proj}\CC[S_n].
\]
The special fibre is the projective toric variety associated with the semigroup cone over \(P_n^\vee\).
\end{thm}

\begin{proof}
The flag is admissible by Lemma~\ref{lem:flag-admissible}, so the associated flag valuation has one-dimensional leaves. The degree-augmented valuation \(s\mapsto(k,\nu_{\mathrm{flag}}(s))\) also has one-dimensional leaves: the degree separates the homogeneous pieces, and within each fixed degree the valuation is the original flag valuation.

For \(0\ne s\in R_k\),
\[
\Sigma\bigl(k,\nu_{\mathrm{flag}}(s)\bigr)
=
\bigl(k,\sigma(\nu_{\mathrm{flag}}(s)-k\mathbf 1)\bigr).
\]
Proposition~\ref{prop:no-body-p2n} identifies the second component, for fixed \(k\), with \(kP_n^\vee\cap\ZZ^d\). Hence \(\Sigma(S_{\mathrm{flag}})=S_n\).

Since \(\Sigma\) is a semigroup isomorphism, it induces an isomorphism of semigroup algebras
\[
\CC[S_{\mathrm{flag}}]\longrightarrow \CC[S_n],
\qquad
\chi^{(k,v)}\longmapsto \chi^{\Sigma(k,v)}.
\]
The semigroup \(S_n\) is the semigroup of lattice points in the rational polyhedral cone over \(P_n^\vee\), so it is finitely generated by Gordan's lemma. Therefore \(S_{\mathrm{flag}}\) is finitely generated as well. Anderson's Newton--Okounkov degeneration theorem applied to the original flag-valuation filtration gives a flat Rees degeneration with general fibre \(R\) and special fibre
\[
\operatorname{gr}_{\nu_{\mathrm{flag}}}R\simeq\CC[S_{\mathrm{flag}}]\simeq\CC[S_n].
\]
Taking \(\operatorname{Proj}\) gives the stated toric degeneration.
\end{proof}

\begin{cor}
Let
\[
W_n:=\left.W_{\mathrm{Lus}}^{(2n-1)}\right|_{q=1}
=
b+
\sum_{m=1}^{n-1}(a_m+c_m)
+
\frac{1}{a_2\cdots a_{n-1}b^2c_{n-1}\cdots c_1}
+
\frac{1}{a_1\cdots a_{n-1}b^2c_{n-1}\cdots c_2},
\]
where an empty product is understood to be \(1\). Then
\[
\operatorname{Newt}(W_n)^\vee=P_n^\vee.
\]
The toric special fibre in Theorem~\ref{thm:no-toric-degeneration} is the projective toric variety whose normal fan is the face fan of \(\operatorname{Newt}(W_n)\).
\end{cor}

\begin{proof}
Expanding the final term of \(W_{\mathrm{Lus}}^{(2n-1)}\) at \(q=1\) gives
\[
\frac{a_1+c_1}
{a_1\cdots a_{n-1}b^2c_{n-1}\cdots c_1}
=
\frac{1}
{a_2\cdots a_{n-1}b^2c_{n-1}\cdots c_1}
+
\frac{1}
{a_1\cdots a_{n-1}b^2c_{n-1}\cdots c_2}.
\]
The exponent vectors of these two monomials, together with those of \(b\), \(a_m\), and \(c_m\), give precisely the inequalities defining \(P_n^\vee\). Hence \(\operatorname{Newt}(W_n)^\vee=P_n^\vee\). The final assertion follows from Theorem~\ref{thm:no-toric-degeneration}.
\end{proof}

\subsection{The boundary degeneration}
There is also a rank-one valuation which degenerates the type \(C_n\) boundary to the standard toric boundary inside the fixed ambient projective space. Choose the weight vector \(\omega=(0^2,1^2,\ldots,N^2)\). For a monomial \(p^\alpha\), set \(\nu_\omega(p^\alpha)=\omega\cdot\alpha\), and for a polynomial take the minimum of the weights of its monomials.

For the \(r\)-th quadratic factor, the term \(p_rp_{N-r}\) has weight \(r^2+(N-r)^2\). The \emph{excess weight} of another term in the same factor means its \(\omega\)-weight minus this candidate minimum weight. For \(1\leq j\leq r\), the term \(p_{r-j}p_{N-r+j}\) has excess weight
\[
(r-j)^2+(N-r+j)^2-r^2-(N-r)^2=2j(N-2r+j)>0.
\]
Every other term has strictly larger weight. Hence \(p_rp_{N-r}\) is the unique initial monomial of \(F_r\), and
\[
\operatorname{in}_\omega(H_N)
=p_0p_N\prod_{r=1}^{n-1}p_rp_{N-r}
=\prod_{i=0}^Np_i.
\]
Equivalently, the Rees family is
\[
H_{N,t}=p_0p_N\prod_{r=1}^{n-1}
\left(p_rp_{N-r}+\sum_{j=1}^r(-1)^jt^{2j(N-2r+j)}p_{r-j}p_{N-r+j}\right).
\]
Then \(H_{N,1}=H_N\), while \(H_{N,0}=\prod_{i=0}^Np_i\). This family is flat. Indeed, \(H_{N,t}\) is homogeneous in the \(p_i\)'s and is not divisible by \(t\), so multiplication by \(t\) is injective on
\[
\CC[p_0,\ldots,p_N,t]/(H_{N,t}).
\]
More generally, \(H_{N,t}\) has a coefficient equal to \(1\) as a polynomial in the \(p_i\)'s, so no nonzero polynomial in \(t\) divides it; hence the quotient has no \(\CC[t]\)-torsion. Since \(\CC[t]\) is a PID, this quotient is \(\CC[t]\)-flat. Equivalently, the hypersurface in \(\PP^N\times\mathbb A^1\) is a relative Cartier divisor with no vertical component. Thus
\[
(\PP^{2n-1},D_C)
\rightsquigarrow
(\PP^{2n-1},D_{\mathrm{tor}})
\]
as pairs, where the ambient variety is fixed but the anticanonical boundary degenerates from the type \(C_n\) boundary to the toric boundary.
This rank-one degeneration is distinct from the Newton--Okounkov toric degeneration of the ambient projective variety. It compares the two boundary divisors inside the fixed ambient projective space.

\end{document}